\documentclass[final,leqno,onefignum,onetabnum]{siamltexmm}
\usepackage{amssymb}
\usepackage{amsmath}
\newtheorem{remark}[theorem]{Remark}
\usepackage{enumerate}

\def\ep{\varepsilon}

\def\th{\theta}

\def\la{\lambda}
\def\rh{\varrho}
\def\si{\sigma}

\def\ph{\varphi}

\def\om{\omega}

\def\De{\Delta}

\def\Om{\Omega}

\let\on=\operatorname

\def\inv{^{-1}}
\def\x{\times}
\def\p{\partial}

\def\R{{\mathbb R}}

\def\Imm{\on{Imm}}

\newcommand{\ud}{\,\mathrm{d}}
\newcommand{\ip}[2]{\langle #1,#2 \rangle}

\title{A Numerical Framework for Sobolev Metrics on the Space of Curves\thanks{Parts of this work have been published in the conference proceedings \cite{BBHM2015a,BBHM2015b}. All authors were partially supported by the Erwin Schr\"{o}dinger Institute programme: Infinite-Dimensional Riemannian Geometry with Applications to Image Matching and Shape Analysis. }} 

\author{Martin Bauer\footnotemark[2] \and Martins Bruveris\footnotemark[3] \and Philipp Harms\footnotemark[4] \and Jakob M{\o}ller-Andersen\footnotemark[5]}

\begin{document}
\maketitle
\newcommand{\slugmaster}{}

\renewcommand{\thefootnote}{\fnsymbol{footnote}}

\footnotetext[2]{Department of Mathematics, Florida State University and Faculty of Mathematics, TU Wien, \url{bauer@math.fsu.edu}. Supported by the FWF project ``Geometry of shape spaces and related infinite dimensional spaces'' (P246251).}
\footnotetext[3]{Department of Mathematics, Brunel University London,  \url{martins.bruveris@brunel.ac.uk}. Supported by the BRIEF award from Brunel University London.}
\footnotetext[4]{Department of Mathematics, Freiburg University, \url{philipp.harms@stochastik.uni-freiburg.de }}
\footnotetext[5]{Department of Applied Mathematics and Computer Science, Technical University of Denmark, \url{jakmo@dtu.dk}}

\renewcommand{\thefootnote}{\arabic{footnote}}

\begin{abstract}
Statistical shape analysis can be done in a Riemannian framework by endowing the set of shapes with a Riemannian metric. Sobolev metrics of order two and higher on shape spaces of parametrized or unparametrized curves have several desirable properties not present in lower order metrics, but their discretization is still largely missing. In this paper, we present algorithms to numerically solve the geodesic initial and boundary value problems for these metrics. The combination of these algorithms enables one to compute Karcher means in a Riemannian gradient-based optimization scheme and perform principal component analysis and clustering. Our framework is sufficiently general to be applicable to a wide class of metrics. We demonstrate the effectiveness of our approach by analyzing a collection of shapes representing HeLa cell nuclei.
\end{abstract}

\begin{keywords}shape analysis, shape registration, Sobolev metric, geodesics, Karcher mean, B-splines\end{keywords}

\begin{AMS}58B20, 58E50 (Primary); 49M25, 68U05 (Secondary)\end{AMS}

\pagestyle{myheadings}
\thispagestyle{plain}
\markboth{BAUER, BRUVERIS, HARMS, M{\O}LLER-ANDERSEN}{A NUMERICAL FRAMEWORK FOR SOBOLEV METRICS ON THE SPACE OF CURVES}

\section{Introduction}
The comparison and analysis of geometric shapes plays a central role in many applications. 
A particularly important class of shapes is the space of curves, 
which is used to model applied problems in medical imaging \cite{XKS2014,Younes2012}, object tracking \cite{Sundaramoorthi2011,Sundaramoorthi2008}, computer animation \cite{Esl2014b,Esl2014}, 
speech recognition \cite{su2014b}, biology \cite{Laga2014,su2014}, and many other fields \cite{Bauer2014,Krim2006}. 

In this article we consider the space $\Imm(S^1,\R^d)$ of closed, regular (or immersed) curves in $\R^d$ as well as some quotients of this space by reparametrizations and Euclidean motions.
These spaces of shapes are inherently nonlinear. To make standard methods of statistical analysis applicable, one can linearize the space locally around each shape. This can be achieved by introducing a Riemannian structure, which describes both the global nonlinearity of the space as well as its local linearity. Over the past decade Riemannian shape analysis has become an active area of research in pure and applied mathematics. Driven by applications, a variety of different Riemannian metrics has been used. 

An important class of metrics are Sobolev metrics. These metrics can be defined initially on the space $\Imm(S^1,\R^d)$ and then induced on quotients of this space by requiring the projections to be Riemannian submersions (see Def.~\ref{def:sobolev_metric} and Thm.~\ref{thm:quotient_metric}). Recently Sobolev metrics of order two were shown to possess much nicer properties than metrics of lower order: the geodesic distance is non-degenerate, the geodesic equation is globally well-posed, any two curves in the same connected component can be connected by a minimizing geodesic, the metric completion consists of all $H^2$-immersions, and the metric extends to a strong Riemannian metric on the metric completion \cite{Bruveris2014b_preprint, Bruveris2014}. 

Numerical methods for the statistical analysis of shapes under second order metrics are, however, still largely missing. This is in contrast to first order metrics, where isometries to simpler spaces led to explicit formulas for geodesics under many parameter configurations of the metric \cite{Bauer2014b, Klassen2004, Jermyn2011, Michor2008a}. For certain $H^2$-metrics an analogous approach was developed in \cite{Bauer2014c}. Moreover, the geodesic boundary value problem under second order Finsler metrics on the space of $BV^2$-curves was implemented numerically in \cite{Vialard2014_preprint}. For general second order Sobolev metrics on spaces of unparametrized curves a numerical framework is, however, still lacking. This is the topic of this paper.  

We present a numerical implementation of the initial and boundary value problems for geodesics under second order Sobolev metrics.\footnote{Our code can be downloaded from \url{https://github.com/h2metrics/h2metrics.git}.} Our implementation is based on a discretization of the Riemannian energy functional using B-splines. The boundary value problem for geodesics is solved by a standard minimization procedure on the set of discretized paths and the initial value problem by discrete geodesic calculus  \cite{Rumpf2014}. Our approach is general in that it allows to factor out reparametrizations and rigid transformations. Moreover, it involves no restriction on the parameters of the metric and could be applied to other, higher-order metrics, as well. 

In future work our framework could be applied to other spaces of mappings like manifold-valued curves, embedded surfaces, or more general spaces of immersions (see \cite{Bauer2011a, Bauer2011b} for details and \cite{Bauer2014} for a general overview). 
\section{Sobolev metrics on spaces of curves}

\subsection{Notation}
The space of smooth, regular curves with values in $\R^d$ is
\begin{align}
\Imm(S^1,\R^d)=\left\{c\in C^{\infty}(S^1,\R^d)\colon \forall \th \in S^1, c'(\th) \neq 0 \right\}\,,
\end{align}
where $\Imm$ stands for \emph{immersions}. We call such curves parametrized to distinguish them from unparametrized curves defined in Sect.~\ref{sec:unparametrized}. The space $\Imm(S^1,\R^d)$ is an open subset of the Fr\'echet space $C^\infty(S^1,\R^d)$ and therefore can be considered as a Fr\'echet manifold. Its tangent space $T_c\Imm(S^1,\R^d)$ at any curve $c$ is the vector space $C^\infty(S^1,\R^d)$ itself. We denote the Euclidean inner product on $\R^d$ by $\langle\cdot,\cdot\rangle$. Differentiation is sometimes denoted using subscripts as in $c_\theta=\partial_\theta c=c'$. Moreover, for any fixed curve $c$, we denote differentiation and integration with respect to arc length by $D_s=\partial_{\theta}/|c_\theta|$ and $\mathrm ds=|c_\theta|\ud \theta$, respectively. A path of curves is a mapping $c\colon\![0,1]\to\Imm(S^1,\R^d)$; its velocity is denoted by $c_t=\partial_t c=\dot c$.

\subsection{Parametrized curves}
\label{sec:parametrized}

In this article we study the following class of weak Riemannian metrics on $\Imm(S^1,\R^d)$.

\begin{definition}
\label{def:sobolev_metric}
A \emph{second order Sobolev metric with constant coefficients} on $\Imm(S^1,\R^d)$ is a weak Riemannian metric of the form
\begin{equation}\label{eq:sobolev_metric}
G_c(h,k) = \int_{S^1} a_0\langle h,k \rangle+a_1\langle D_s h,D_s k \rangle+a_2\langle D_s^2 h,D_s^2 k \rangle \ud s \,,
\end{equation}
where $h,k \in T_c\Imm(S^1,\R^d)$, and $a_j \in \R$ are constants with $a_0, a_2 > 0$ and $a_1 \geq 0$. If $a_2=0$ and $a_1 > 0$ it is a first order metric and if $a_1=a_2=0$ it is a zero order or $L^2$-metric.
\end{definition}

Note that the symbols $D_s$ and $\ud s$ hide the dependency of the Riemannian metric on the base point $c$. Expressing derivatives in terms of $\theta$ instead of arc length, one has
\begin{equation}\label{def:sobolev_metric2}
\begin{aligned}
 G_c(h,k)&= \int_0^{2\pi} a_0 |c'| \langle h, k \rangle + \frac{a_1}{|c'|} \langle h', k' \rangle
+ \frac{a_2}{|c'|^7} \langle c', c'' \rangle^2 \langle h', k' \rangle \\
&\qquad
- \frac{a_2}{|c'|^5} \langle c', c'' \rangle \big( \langle h', k'' \rangle + \langle h'', k' \rangle \big) +
\frac{a_2}{|c'|^3} \langle h'', k'' \rangle
\ud \th\,.
\end{aligned}
\end{equation}

In the Riemannian setting the length of a path $c\colon[0,1]\to\Imm(S^1,\R^d)$ is defined as
\begin{align}
L(c) = \int_0^1 \sqrt{G_{c(t)}(c_t(t),c_t(t))} \ud t\,,
\end{align}
and the geodesic distance between two curves $c_0,c_1\in\Imm(S^1,\R^d)$ is the infimum of the lengths of all paths connecting these curves, i.e., 
\[
\on{dist}(c_0, c_1) = \inf_c\left\{L(c)\colon c(0)=c_0,\, c(1)=c_1\right\}\,.
\]
On finite-dimensional manifolds the topology induced by the geodesic distance coincides with the manifold topology by the Hopf--Rinow theorem. On infinite-dimensional manifolds with weak Riemannian metrics this is not true anymore. For example, the geodesic distance induced by the $L^2$-metric on curves vanishes identically \cite{Bauer2012c, Michor2006c}. On the other hand, first and second order metrics overcome this degeneracy, as the following result of \cite{Michor2006c,Michor2007} shows.

\begin{theorem}
The geodesic distance of first and second order metrics on $\Imm(S^1,\R^d)$ separates points, i.e., $\on{dist}(c_0, c_1)>0$ holds for all $c_0\neq c_1$.
\end{theorem}

Geodesics are locally distance-minimizing paths. They can be described by a partial differential equation, called the geodesic equation. It is the first order condition for minima of the energy functional
\begin{equation} \label{eq:EnergyFunctional}
\begin{aligned}
E(c) &= \int_0^1 G_{c(t)}(c_t(t), c_t(t)) \ud t \,.
\end{aligned}
\end{equation}
Recently some local and global existence results for geodesics of Sobolev metrics were shown in \cite{Bruveris2014, Bruveris2014b_preprint, Michor2007}. We summarize them here since they provide the theoretical underpinnings for the numerical methods presented in this paper.

\begin{theorem}
\label{thm:long_time}
The geodesic equation of second order metrics, written in terms of the momentum 
$p=|c'| (a_0 c_t-a_1 D^2_s c_t+a_2 D_s^4c_t)$, is given by
\begin{align}
\partial_t p = &-\frac{a_0}2 |c_\theta| D_s(\langle c_t,c_t \rangle D_sc) + \frac{a_1}{2}|c_\theta| D_s (\langle D_s c_t,D_s c_t \rangle D_sc) \nonumber
\\&
-\frac{a_2}{2}|c_\theta|D_s (\langle D^3_s c_t,D_s c_t \rangle D_sc)+\frac{a_2}{2}|c_\theta|D_s (\langle D^2_s c_t,D^2_s c_t \rangle D_sc)\,.
\label{eq:GeodesicEquation}
\end{align}
For any initial condition $(c_0, u_0) \in T\Imm(S^1,\R^d)$ the geodesic equation has a unique solution, which exists for all time. In contrast, the geodesic equation of first order Sobolev metrics is locally, but not globally, well-posed. 
\end{theorem}

\emph{\begin{remark}
\emph{~The choice of parameters $a_0$, $a_1$, and $a_2$ of the Riemannian metric can have a large influence on the resulting optimal deformations. We illustrate this in Fig.~\ref{fig:geodesicsConstants}, where we show the geodesic  between a fish-like and a tool-like curve for various choices of parameters. 
}\end{remark}}

\begin{figure}
\centering
	\includegraphics[width=.2\textwidth]{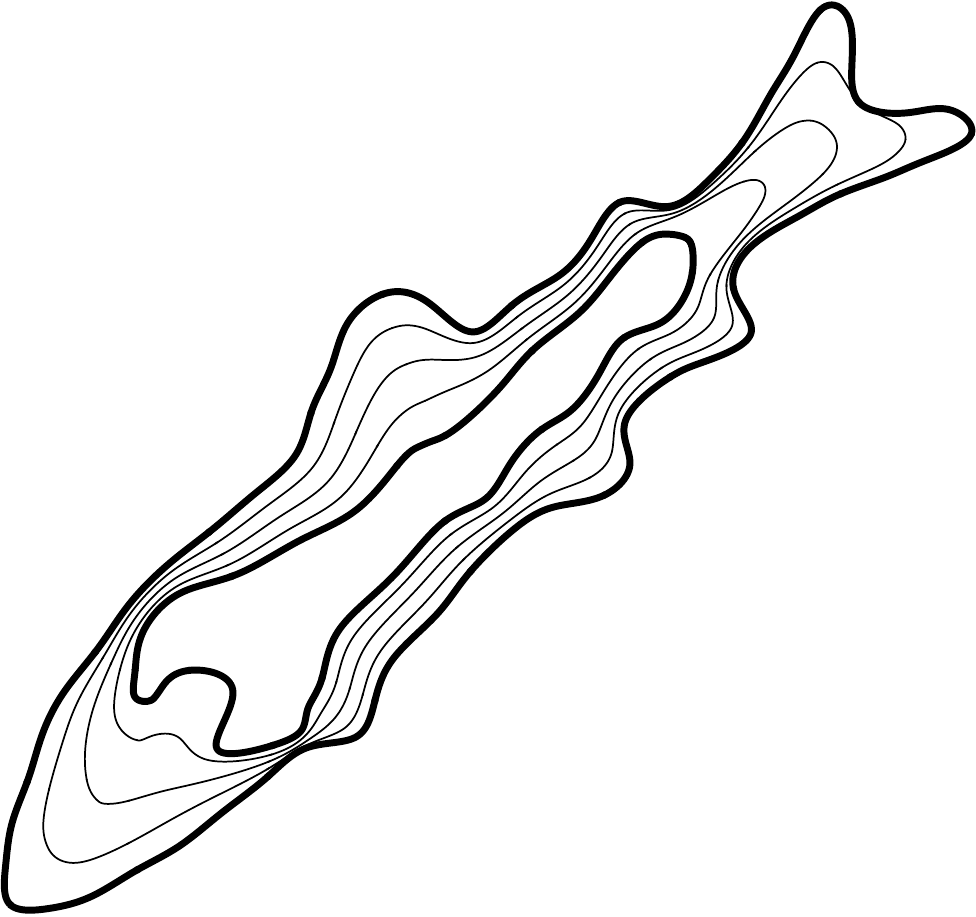}
	\includegraphics[width=.2\textwidth]{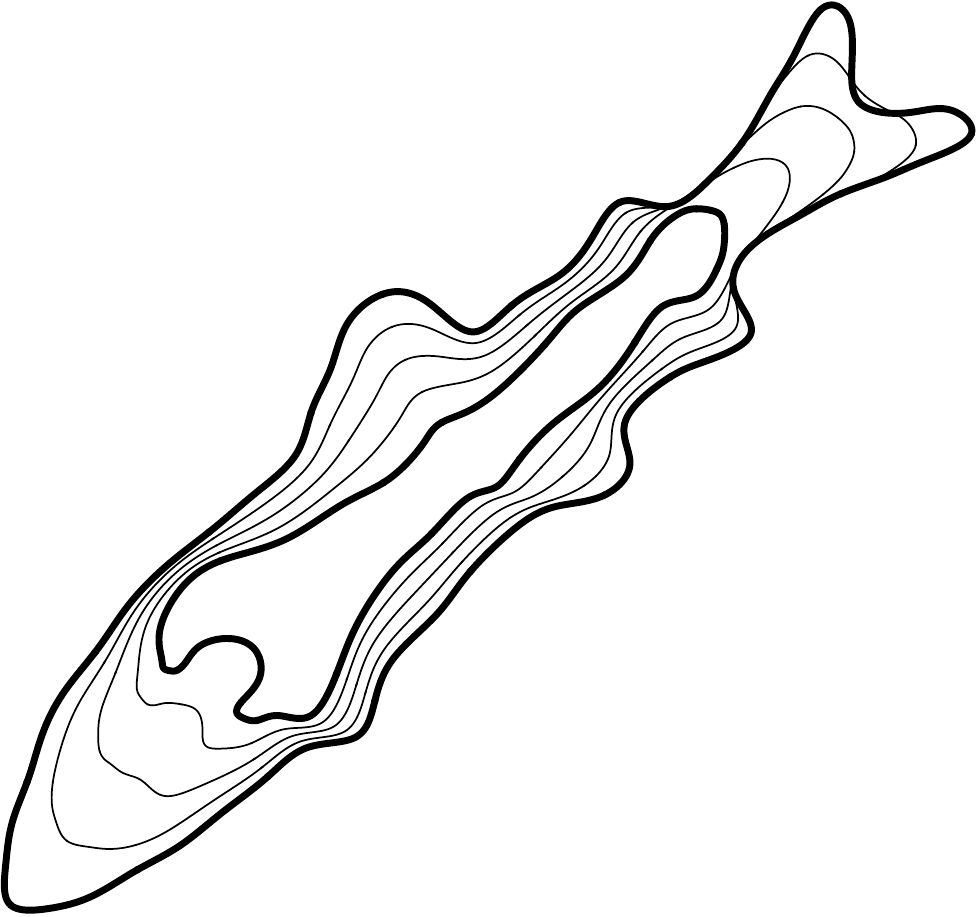}	
	\includegraphics[width=.2\textwidth]{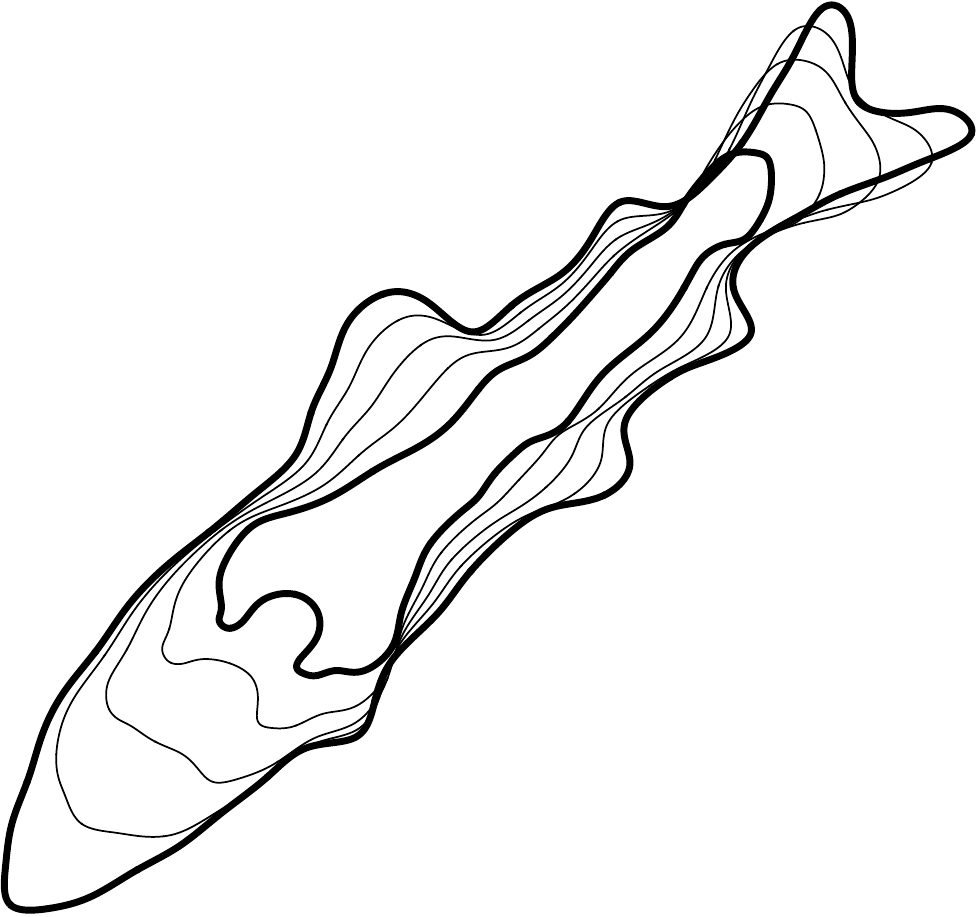}		
	\includegraphics[width=.2\textwidth]{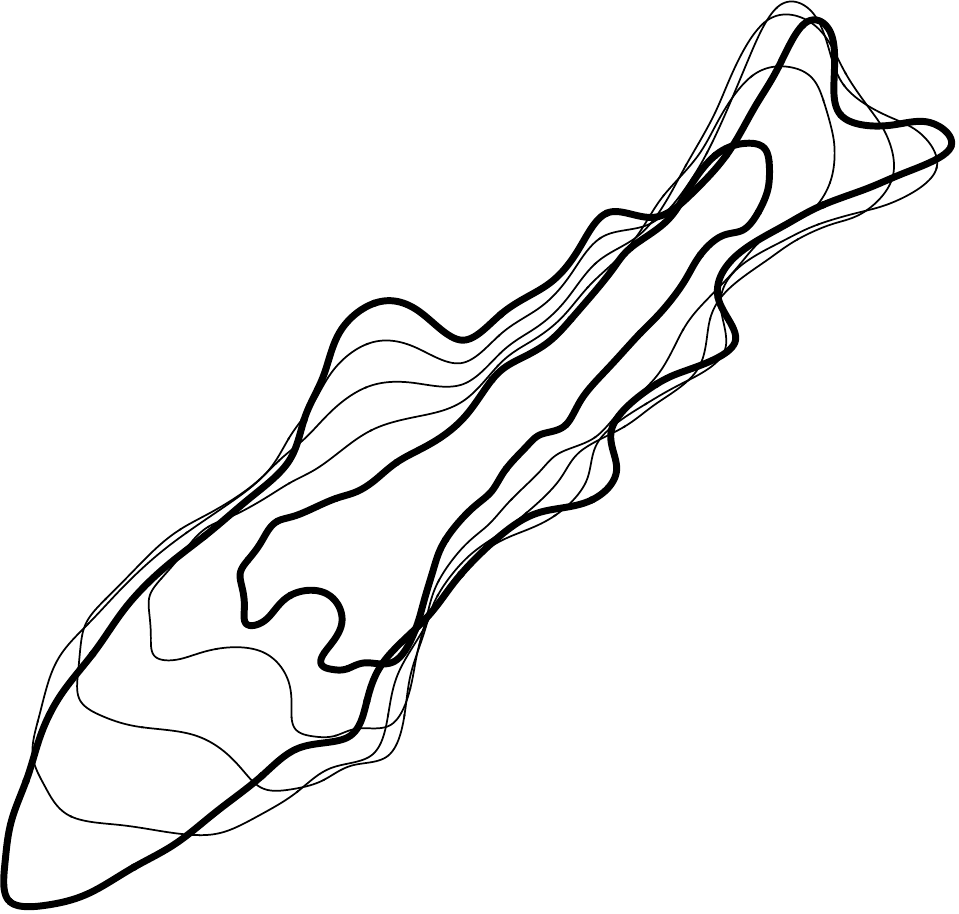}
	\caption{Influence of the constants in the metric on geodesics between a fish and a tool in the space of unparametrized curves. The metric parameter $a_1$ is set to zero, whereas the parameter $a_2$ is increased by a factor 10 in the second, a factor 100 in the third, and a factor 1000 in the fourth column. The corresponding geodesic distances are $135.65$, $162.35$, $229.26$ and $451.9$. Note that since we also optimize over translations and rotations of the target curve, the position in space varies.}
	\label{fig:geodesicsConstants}
\end{figure}

For second order metrics it is possible to compute the metric completion of the space of smooth immersions. We introduce the Banach manifold of Sobolev immersions
\begin{equation}
\mathcal I^2(S^1,\R^d)=\left\{c\in H^2(S^1,\R^d)\colon \forall \th \in S^1, c'(\th) \neq 0 \right\}\,.
\end{equation}
By the Sobolev embedding theorem this space is well-defined and an open subset of the space of all $C^1$-immersions.
It has been shown in \cite{BH2015,Bruveris2014b_preprint} that $\mathcal I^2(S^1,\R^d)$ coincides with the metric completion of the space of smooth immersions:
\begin{theorem}\label{thm:completenessImm}
The metric completion of the space $\Imm(S^1,\R^d)$ endowed with a second order Sobolev metric is $\mathcal I^2(S^1,\R^d)$. Furthermore, any two two curves $c_0$ and $c_1$ in the same connected component of 
$\mathcal I^2(S^1,\R^d)$ can be joined by a minimizing geodesic.
\end{theorem}

\subsection{Unparametrized curves}
\label{sec:unparametrized}

In many applications curves are considered equal if they differ only by their parametrization, i.e., we identify the curves $c$ and $c \circ \ph$, where $\ph \in \on{Diff}(S^1)$ is a reparametrization. The reparametrization group $\on{Diff}(S^1)$ is the diffeomorphism group of the circle,
\begin{equation*}
\on{Diff}(S^1) = \left\{ \ph \in C^\infty(S^1,S^1) \colon \ph' > 0 \right\}\,,
\end{equation*}
which is an infinite-dimensional regular Fr\'echet Lie group \cite{Michor1997}. Reparametrizations act on curves by composition from the right, i.e., $c\circ \ph$ is a reparametrization of $c$. The space 
\begin{equation*}
 B_i(S^1,\R^d)=\Imm(S^1,\R^d) / \on{Diff}(S^1)\,,
\end{equation*}
of unparametrized curves is the orbit space of this group action. This space is not a manifold; it has singularities at any curve $c$ with nontrivial isotropy subgroup \cite{Michor1991}. We therefore restrict ourselves to the dense open subset $\Imm_{f}(S^1,\R^d)$ of curves upon which $\on{Diff}(S^1)$ acts freely and define 
\begin{equation*}
 B_{i,f}(S^1,\R^d) = \Imm_{f}(S^1,\R^d) / \on{Diff}(S^1)\,.
\end{equation*}
This restriction, albeit important for theoretical reasons, has no influence on the practical applications of Sobolev metrics, since $B_{i,f}(S^1,\R^d)$ is open and dense in $B_i(S^1,\R^d)$. We have the following result concerning the manifold structure of the orbit space and the descending properties of Sobolev metrics \cite{Bauer2011b, Michor1991, Michor2007}.

\begin{theorem}\label{thm:shape_space} 
The space $ B_{i,f}(S^1,\R^d)$ is a Fr\'echet manifold and the base space of the principal fibre bundle
\begin{equation*}
\pi: \Imm_f(S^1,\R^d) \to  B_{i,f}(S^1,\R^d)\,,\quad c \mapsto c \circ \on{Diff}(S^1)\,,
\end{equation*}
with structure group $\on{Diff}(S^1)$. A Sobolev metric $G$ on $\Imm_f(S^1,\R^d)$ induces a metric on $B_{i,f}(S^1,\R^d)$ such that the projection $\pi$ is a Riemannian submersion.
\end{theorem}

The induced Riemannian metric on $B_{i,f}(S^1,\R^d)$ defines a geodesic distance, which can also be calculated using paths in $\on{Imm}_f(S^1,\R^d)$ connecting $c_0$ to the orbit $c_1 \circ \on{Diff}(S^1)$, i.e., for $\pi(c_0),\pi(c_1) \in B_{i,f}(S^1,\R^d)$ we have,
\begin{equation*}
\on{dist}\big(\pi(c_0), \pi(c_1)\big) = \inf \left\{ L(c)\colon c(0) = c_0,\, c(1) \in c_1 \circ \on{Diff}(S^1)\right\} \,.	
\end{equation*}
To relate the geometries of $\on{Imm}(S^1,\R^d)$ and $B_{i,f}(S^1,\R^d)$, one defines the vertical and horizontal subspaces of $T_c\on{Imm}_f(S^1,\R^d)$,
\[
\on{Ver}_c = \on{ker}(T_c \pi)\,,\quad
\on{Hor}_c = \left(\on{Ver}_c\right)^{\perp,G_c}\,.
\]
As shown in \cite{Michor2007} they form a decomposition of $T_c\on{Imm}_f(S^1,\R^d)$,
\begin{align*}
T_c\Imm_f(S^1,\R^d)&=\on{Ver}_c\oplus \on{Hor}_c\,,
\end{align*}
as a direct sum. More explicitly, 
\begin{align*}
\on{Ver}_c&=\left\{g.v_c\in  T_c\Imm_f(S^1,\R^d)\colon g\in C^{\infty}(S^1)\right\} \\
\on{Hor}_c&=\left\{k\in  T_c\Imm_f(S^1,\R^d)\colon 
\langle a_0 k-a_1 D^2_s k+a_2 D_s^4k,v_c\rangle=0 \right\}\,,
\end{align*}
with $v_c = D_s c$ the unit tangent vector to $c$.

A geodesic $c$ on $\on{Imm}_f(S^1,\R^d)$ is called \emph{horizontal} at $t$, if $\p_t c(t) \in \on{Hor}_{c(t)}$. It can be shown that if $c$ is horizontal at $t=0$, then it is horizontal at all $t$. Furthermore, geodesics on $B_{i,f}(S^1,\R^d)$ can be lifted to horizontal geodesics on $\Imm_f(S^1,\R^d)$ and the lift is unique if we specify the initial position of the lift; conversely, horizontal geodesics on $\on{Imm}(S^1,\R^d)$ project down to geodesics on $B_{i,f}(S^1,\R^d)$.

What about long-time existence of geodesics on $B_{i,f}(S^1,\R^d)$? Using the correspondence between geodesics on $B_{i,f}(S^1,\R^d)$ and horizontal geodesics on $\on{Imm}_f(S^1,\R^d)$ together with Thm.~\ref{thm:long_time} we see that the horizontal lift of a geodesic can be extended for all times. However, it can leave the subset of free immersions and pass through curves with a non-trivial isotropy group. Thus the space $B_{i,f}(S^1,\R^d)$ is not geodesically complete, but we can regain geodesic completeness if we allow geodesics to pass through $B_i(S^1,\R^d)$.

The space $B_{i}(S^1,\R^d)$ inherits some of the completeness properties of $\Imm(S^1,\R^d)$. To formulate these properties we introduce the group $\mathcal D^2(S^1)$ of $H^2$-diffeomorphisms and the corresponding shape space of Sobolev immersions,
\begin{equation*}
\mathcal B^2(S^1,\R^d)= \mathcal I^2(S^1,\R^d)/\mathcal D^2(S^1)\,.
\end{equation*}
It is not known whether this space is a smooth Banach manifold, it is however a metric length space. The structure of it is explained in more detail in the article \cite{Bruveris2014b_preprint}, where the following completeness result is proven.
\begin{theorem}\label{thm:quotient_metric}
 Let $G$ be a second order Sobolev metric with constant coefficients.
\begin{enumerate}[(1)]
\item
The space $\left(\mathcal B^2(S^1,\R^d),\on{dist}\right)$, where $\on{dist}$ is the quotient distance induced by $\left(\mathcal I^2(S^1,\R^d),\on{dist}\right)$, is a complete metric space, and it is the metric completion of $\left(B_{i,f}(S^1,\R^d), \on{dist}\right)$.

\item
Given two unparametrized curves $C_1, C_2 \in \mathcal B^2(S^1,\R^d)$ in the same connected component, there exist $c_1, c_2 \in 
\mathcal I^2(S^1,\R^d)$ with $c_1 \in \pi\inv(C_1)$ and $c_2 \in \pi\inv(C_2)$, such that
\[
\on{dist}(C_1, C_2) = \on{dist}(c_1, c_2)\,;
\]
equivalently the infimum in
\[
\on{dist}(\pi(c_1), \pi(c_2)) = \inf_{\ph \in \mathcal D^2(S^1)} \on{dist}(c_1, c_2 \circ \ph)
\]
is attained.

\item
The metric space $\left(\mathcal B^2(S^1,\R^d),\on{dist}\right)$ is a length space and any two shapes in the same connected component can be joined by a minimizing geodesic.
\end{enumerate}
\end{theorem}

In the last statement of the above theorem we have to understand a minimizing geodesic in the sense of metric spaces.

\subsection{Euclidean motions}

Curves modulo Euclidean motions are a natural object of consideration in many applications. The Euclidean motion group $SE(d) = SO(d) \ltimes \R^d$ is the semi-direct product of the translation group $\R^d$ and the rotation group $SO(d)$. These groups act on $\Imm(S^1,\R^d)$ by composition from the left. The metric \eqref{def:sobolev_metric2} is invariant under these group actions,
\begin{align*}
G_{R.c+a}(R.h,R.k)=G_{c}(h,k)\qquad\forall (R,a) \in SE(d)\,.
\end{align*}
As in the previous section we obtain an induced Riemannian metric on the quotient space 
\[
\mathcal S(S^1,\R^d)=\Imm_f(S^1,\R^d)/\on{Diff}(S^1)\x SE(d) = B_{i,f}(S^1,\R^d) / SE(d) \,,
\]
such that the projection $\pi : \on{Imm}_f(S^1,\R^d) \to \mathcal S(S^1,\R^d)$ is a Riemannian submersion:
\begin{theorem}The space $\mathcal S(S^1,\R^d)$ is a Fr\'echet manifold and the base space of the principal fibre bundle
\begin{equation*}
\pi: \Imm_f(S^1,\R^d) \to  \mathcal S(S^1,\R^d)\,,\quad c \mapsto SE(d).c \circ \on{Diff}(S^1) \,,
\end{equation*}
with structure group $\on{Diff}(S^1) \x SE(d)$. A Sobolev metric $G$ on $\Imm_f(S^1,\R^d)$ induces a metric on $\mathcal S(S^1,\R^d)$ such that the projection $\pi$ is a Riemannian submersion.
\end{theorem}

Note that the left action of $SE(d)$ commutes with the right action of $\on{Diff}(S^1)$ and hence the order of the quotient operations does not matter.  The induced geodesic distance on the quotient space is given by the infimum
\begin{equation*}
\on{dist}\big(\pi(c_0), \pi(c_1)\big) = \inf \left\{ L(c)\colon c(0) = c_0,\, c(1) \in \pi(c_1) = SE(d) .c_1 \circ \on{Diff}(S^1 )\right\} \,,
\end{equation*}
with the infimum being taken over paths $c : [0, 1] \to \on{Imm}(S^1,\R^d)$.

Similarily as in the previous section geodesics on $\mathcal S(S^1,\R^d)$ can be lifted to horizontal geodesics on $\Imm_f(S^1,\R^d)$ and, conversely, horizontal geodesics on $\on{Imm}(S^1,\R^d)$ project down to geodesics on $ \mathcal S(S^1,\R^d)$. Thus the space $\mathcal S(S^1,\R^d)$ inherits again some of the completeness properties of $\Imm(S^1,\R^d)$ and we obtain the equivalent of Thm.~\ref{thm:quotient_metric} also for the space $\mathcal B^2(S^1,\mathbb R^d)/SE(d)$.

\emph{\begin{remark}
\emph{~The Sobolev metric \eqref{def:sobolev_metric2} is not invariant with respect to scalings. However, this lack of invariance can be addressed by introducing weights depending on the length $\ell_c$ of the curve $c$. The modified metric
\begin{equation*}
\widetilde G_c(h,k) = \int_{S^1} \frac{a_0}{\ell^3_c}\langle h,k \rangle+\frac{a_1}{\ell_c} \langle D_s h,D_s k \rangle+a_2\ell_c \langle D_s^2 h,D_s^2 k \rangle \ud s
\end{equation*}
is invariant with respect to scalings. It induces a metric on the quotient space $\mathcal S(S^1,\R^d)/\mathbb R_+$, where $\mathbb R_+$ is the scaling group acting by multiplication $(\la, c) \mapsto \la.c$ on curves.
}\end{remark}}

\section{Discretization}

In order to numerically compute geodesics, the infinite-dimensional space of curves must be discretized. The method we choose is standard: we construct an appropriate finite-dimensional function space and perform optimization therein. We choose B-splines among the many possible options because B-splines and their derivatives have piecewise polynomial representations and can be evaluated efficiently. This permits fast and simple computation of the energy functional and its derivatives. Furthermore, in contrast to standard finite-element discretization, it is possible to control the global regularity of the functions. For details regarding B-splines, their definition, efficient computations, etc., we refer to \cite{Schumaker2007} and the vast literature on the subject. 

For simplicity, we shall work only with \textit{simple} B-splines, i.e., splines where all interior knots have multiplicity one. Hence the splines have maximal regularity at the knots. We will define splines of degrees $n_t$ and $n_\theta$ in the variables $t \in [0,1]$ and $\theta\in[0,2\pi]$, respectively. The corresponding numbers of control points are denoted by $N_t$ and $N_\theta$. For $t$ we use a uniform knot sequence on the interval $[0,1]$ with full multiplicity at the boundary knots:
\[
\De_t = \{ t_i \}_{i=0}^{2n_t+N_t}\,,\qquad
t_i = \begin{cases} 
0 & 0 \leq i < n_t \\
\displaystyle\frac{i-n_t}{N_t} & n_t \leq i < n_t + N_t \\
1 & n_t + N_t \leq i \leq 2n_t + N_t\,.
\end{cases}
\]
For $\th$ we want the splines to be periodic on the interval $[0,2\pi]$. Therefore we choose knots
\[
\De_\th = \{ \th_j \}_{j=0}^{2n_\th + N_\th}\,,\qquad
\th_j = \frac{j-n_\th}{2\pi N_\th}\,,\quad 0 \leq j \leq 2n_\th+N_\th\,.
\]
The corresponding normalized B-spline basis functions are denoted by $B_i(t)$ and $C_j(\theta)$. Note that all interior knots have multiplicity one, i.e., the splines are simple. Therefore, they have maximal regularity at the knots,
\begin{align*}
B_i \in C^{n_t-1}([0,1])\,, \qquad C_j \in C^{n_\theta-1}(S^1)\,, \qquad i=1,\dots,N_t\,, \qquad j=1,\dots,N_\theta\,.
\end{align*}
Let $\mathcal S^{n_t}_{N_t}$ denote the orthogonal projection from $H^{n_t}([0,1])$ onto the span of the basis functions $B_i$. Similarly, let $\mathcal S^{n_\theta}_{N_\theta}$ denote the orthogonal projection from $H^{n_\theta}(S^1)$ onto the span of the basis functions $C_j$. Then
\begin{align*}
\lim_{N_t\to\infty} \| \mathcal S^{n_t}_{N_t} f-f\|_{H^{n_t}([0,1])}=0\,, \qquad
\lim_{N_\theta\to\infty} \| \mathcal S^{n_\theta}_{N_\theta} g-g\|_{H^{n_\theta}(S^1)}=0\,,
\end{align*}
holds for each $f \in H^{n_t}([0,1])$ and each $g \in H^{n_\theta}(S^1)$. This is a well-known result on the approximation power of one-dimensional splines (c.f. Lem.~\ref{lem:onesplines}); a detailed analysis can be found in \cite{Schumaker2007}. 

The generalization of this statement to multiple dimensions involves tensor product splines and mixed-order Sobolev spaces. Tensor product splines are linear combinations of $B_i\otimes C_j$, where the basis functions $B_i$ are interpreted as functions of $t$ and $C_j$ as functions of $\theta$. To be explicit, a path of curves is represented as a tensor product B-spline with control points $c_{i,j} \in \R^d$ as follows:
\begin{equation}
\label{eq:TensorPath}
c(t, \th) = \sum_{i=1}^{N_t} \sum_{j=1}^{N_\th} c_{i,j} B_i(t) C_j(\th)\,.
\end{equation}
Sobolev spaces of mixed order are Hilbert spaces defined for each $k,\ell\in \mathbb N$ as
\begin{equation}\label{Eq:mixedorderSob}\begin{aligned}
H^{k,\ell}([0,1]\times S^1) = \big\{ f \in L^2([0,1]\times S^1): \exists f^{(k,0)},f^{(0,\ell)},f^{(k,\ell)} \in L^2([0,1]\times S^1)\big\}\,,\\
\langle f,g \rangle_{H^{k,\ell}} = 
\langle f,g\rangle_{L^2} 
+ \langle f^{(k,0)},g^{(k,0)}\rangle_{L^2} 
+ \langle f^{(0,\ell)},g^{(0,\ell)}\rangle_{L^2} 
+ \langle f^{(k,\ell)},g^{(k,\ell)}\rangle_{L^2} \,.
\end{aligned}
\end{equation}
Function spaces of this type were first defined in \cite{nikol1962boundary,nikol1965stable}. We refer to \cite{vybiral2006function} and  \cite{schmeisser2007recent} for detailed expositions and further references.  As before we define for each number of control points $N_t,N_\theta$ the spline approximation operator $\mathcal{S}^{n_t,n_\theta}_{N_t,N_\theta}$ to be the orthogonal projection from $H^{n_t,n_\theta}([0,1]\times S^1)$ onto the span of the tensor product splines $B_i\otimes C_j$. It can be shown that $\mathcal{S}^{n_t,n_\theta}_{N_t,N_\theta} = \mathcal{S}^{n_t}_{N_t} \otimes \mathcal S^{n_\theta}_{N_\theta}$.
\begin{lemma}\label{lem:tensorspline}
For each $n_t \geq k, n_\th\geq \ell$ and each $c \in H^{k,\ell}([0,1]\times S^1)$, 
\begin{align*}
\lim_{N_t,N_\theta\to\infty} \| c-\mathcal{S}^{n_t,n_\theta}_{N_t,N_\theta}c\|_{H^{k,\ell}([0,1]\times S^1)} = 0\,.
\end{align*} 
\end{lemma}
The lemma is proven in App.~\ref{app:convergence} by showing that $H^{n_t,n_\theta}([0,1]\times S^1)$ is isometrically isomorphic to the  Hilbert space tensor product of $H^{n_t}([0,1])$ and $H^{n_\theta}(S^1)$.

\subsection{Discretization of the energy functional}
\label{sec:EnergyFunctional} 

The energy of a path of curves $c:[0,1]\times S^1\to\R^d$ is given by 
\begin{equation}
\label{eq:energyH2_explicit}
\begin{aligned}
E(c) = \int_0^1 G_c(\dot c, \dot c) \ud t 
&= \int_0^1 \int_0^{2\pi} a_0 |c'| \langle \dot c, \dot c \rangle + \frac{a_1}{|c'|} \langle \dot c', \dot c' \rangle
+ \frac{a_2}{|c'|^7} \langle c', c'' \rangle^2 \langle \dot c', \dot c' \rangle \\
&\qquad\quad
- \frac{2a_2}{|c'|^5} \langle c', c'' \rangle \langle \dot c', \dot c'' \rangle +
\frac{a_2}{|c'|^3} \langle \dot c'', \dot c'' \rangle
\ud \th \ud t\,,
\end{aligned}
\end{equation}
as can be seen by combining~\eqref{def:sobolev_metric2} and~\eqref{eq:EnergyFunctional}. 
In the following let $U$ denote the set of all paths $c \in H^{1,2}([0,1]\times S^1;\mathbb R^d)$ with nowhere vanishing spatial derivative, i.e., $c'(t,\theta)=\partial_\theta c(t,\theta)\neq 0$ holds for all $(t,\theta)\in[0,1]\times S^1$. Then $U$ is an open subset of $H^{1,2}([0,1]\times S^1;\mathbb R^d)$ because $H^{1,2}([0,1]\times S^1;\mathbb R^d)$ embeds continuously into $C^{0,1}([0,1]\times S^1;\mathbb R^d)$ by Lem.~\ref{lem:embedding}. The following lemma shows that the energy of a spline tends to the energy of the approximated curve as the number of control points tends to infinity.

\begin{lemma}\label{lem:continuity}
If $n_t\geq 1$ and $n_\theta\geq 2$, then 
\[
\lim_{N_t,N_\theta\to\infty} E(\mathcal S^{n_t,n_\theta}_{N_t,N_\theta} c)= E(c)
\]
holds for each $c \in U$. 
\end{lemma}

\begin{proof}
By Lem.~\ref{lem:tensorspline} the spline approximations $\mathcal{S}^{n_t,n_\theta}_{N_t,N_\theta}c$ converge to $c$ in $H^{1,2}([0,1]\times S^1)$. As $U$ is open, $E(\mathcal S^{n_t,n_\theta}_{N_t,N_\theta} c)$ is well-defined for $N_t,N_\theta$ sufficiently large. The convergence $E(\mathcal S^{n_t,n_\theta}_{N_t,N_\theta} c)\to E(c)$ follows from the $H^{1,2}$-continuity of the energy functional. 
\end{proof}

To discretize the integrals in the definition of the energy functional we use Gaussian quadrature with $m_t$ and $m_\th$ quadrature points on each interval between consecutive knots. The total number of quadrature points is therefore $M_t = m_tN_t$ in time and $M_\th = m_\th N_\th$ in space, and the discrete approximations of the Lebesgue measures on $[0,1]$ and $S^1$ are
\begin{align*}
\mu^{m_t}_{N_t} = \sum_{i=1}^{M_t} w_i \delta_{\bar t_i}, \qquad
\nu^{m_\theta}_{N_\theta} = \sum_{j=1}^{M_\theta} \om_j \delta_{\bar \theta_j}, 
\end{align*}
where $w_i, \om_j$ are the Gaussian quadrature weights and $\bar t_i, \bar \theta_j$ the Gaussian quadrature points. We define the discretized energy $E^{m_t,m_\theta}_{N_t,N_\theta}(c)$ of a curve $c \in C^{1,2}([0,1]\times S^1)\cap U$ to be given by the right-hand side of \eqref{eq:energyH2_explicit} with $\mathrm dt \ud \theta$ replaced by $\mu^{m_t}_{N_t}(\mathrm dt)\nu^{m_\theta}_{N_\theta}(\mathrm d\theta)$. The following theorem shows that the discretized energy of a path tends to the energy of the approximated path as the number of control points go to infinity, provided that the path is smooth enough.

\begin{theorem}
If $n_t\geq 2$, $n_\theta\geq 3$, and $m_t,m_\theta\geq 1$, then 
\begin{equation*}
\lim_{N_t,N_\theta\to\infty} E^{m_t,m_\theta}_{N_t,N_\theta}(\mathcal S^{n_t,n_\theta}_{N_t,N_\theta} c)= E(c)
\end{equation*}
holds for each $c \in U\cap H^{2,3}([0,1]\times S^1)$. 
\end{theorem}

\begin{proof}
The total error can be decomposed into a spline approximation error and a quadrature error:
\begin{equation}\label{equ:triangle}
|E^{m_t,m_\theta}_{N_t,N_\theta}(\mathcal S^{n_t,n_\theta}_{N_t,N_\theta} c)-E(c)|
\leq
|E^{m_t,m_\theta}_{N_t,N_\theta}(\mathcal S^{n_t,n_\theta}_{N_t,N_\theta} c)-E^{m_t,m_\theta}_{N_t,N_\theta}(c)|+|E^{m_t,m_\theta}_{N_t,N_\theta}(c)-E(c)|\,.
\end{equation}
To show that the first summand on the right-hand side tends to zero, note that the spline approximations $\mathcal S^{n_t,n_\theta}_{N_t,N_\theta} c$ converge to $c$ in $H^{2,3}([0,1]\times S^1)$ by Lem.~\ref{lem:tensorspline}. They also converge in $C^{1,2}([0,1]\times S^1)$ by Lem.~\ref{lem:embedding}. Let $F(c)$ denote the integrand in \eqref{eq:energyH2_explicit}. Then $F$ is locally Lipschitz continuous when seen as a mapping from $U\cap C^{1,2}([0,1]\times S^1)$ to $C([0,1]\times S^1)$. Let $L$ denote the Lipschitz constant of $F$ near $c$. Then the first summand in \eqref{equ:triangle} can be estimated for sufficiently large $N_t,N_\theta$ via
\begin{equation*}\begin{aligned}
|E^{m_t,m_\theta}_{N_t,N_\theta}(\mathcal S^{n_t,n_\theta}_{N_t,N_\theta} c)-E^{m_t,m_\theta}_{N_t,N_\theta}(c)|
&\leq
\iint |F(\mathcal S^{n_t,n_\theta}_{N_t,N_\theta} c)-F(c)|\mu^{m_t}_{N_t}\nu^{m_\theta}_{N_\theta}
\\&\leq 
L \|\mathcal S^{n_t,n_\theta}_{N_t,N_\theta} c-c\|_{C^{1,2}([0,1]\times S^1)} \to 0\,.
\end{aligned}\end{equation*}
It remains to show that the second summand in \eqref{equ:triangle} tends to zero. As the Gaussian quadrature rules $\mu^{m_t}_{N_t}$ and $\nu^{m_\theta}_{N_\theta}$ are of order $m_t,m_\theta\geq 1$, there is $K>0$ such that the following estimates hold for all $f \in C^1([0,1])$ and $g\in C^1(S^1)$:
\begin{equation*}
\int_{[0,1]}f(t)\big(\mu^{m_t}_{N_t}(\mathrm dt)-\mathrm dt\big) \leq K N_t^{-1} \|f'\|_{C([0,1])}\,, \quad
\int_{S^1}g(\theta)\big(\nu^{m_\theta}_{N_\theta}(\mathrm d\theta)-\mathrm d\theta\big) \leq K N_\theta^{-1} \|g'\|_{C(S^1)}\,. 
\end{equation*}
See e.g. \cite[Thm.~4.3.1]{brass2011quadrature} for this well-known result. Therefore, the second summand in \eqref{equ:triangle} satisfies
\begin{equation}\begin{aligned}
&|E^{m_t,m_\theta}_{N_t,N_\theta}(c)-E(c)|
=
\left|\iint F(c)(t,\theta)(\mu^{m_t}_{N_t}(\mathrm dt)\nu^{m_\theta}_{N_\theta}(\mathrm d\theta)-\mathrm d t\ud \theta)\right|
\\&\qquad\leq
\left|\iint F(c)(t,\theta)\big(\mu^{m_t}_{N_t}(\mathrm dt)-\mathrm dt\big)\nu^{m_\theta}_{N_\theta}(\mathrm d\theta)\right|
+\left|\iint F(c)(t,\theta)\ud t \big(\nu^{m_\theta}_{N_\theta}(\mathrm d\theta)-\mathrm d\theta\big)\right|
\\&\qquad\leq
K N_t^{-1}\|\partial_t F(c)\|_{C([0,1]\times S^1)} 
+ K N_\theta^{-1}\|\partial_\theta F(c)\|_{C([0,1]\times S^1)}  \to 0\,.
\end{aligned}\end{equation}
This shows that the total error \eqref{equ:triangle} tends to zero as $N_t,N_\theta$ tend to infinity.
\end{proof}

\begin{figure}
\includegraphics[width=0.12\textwidth]{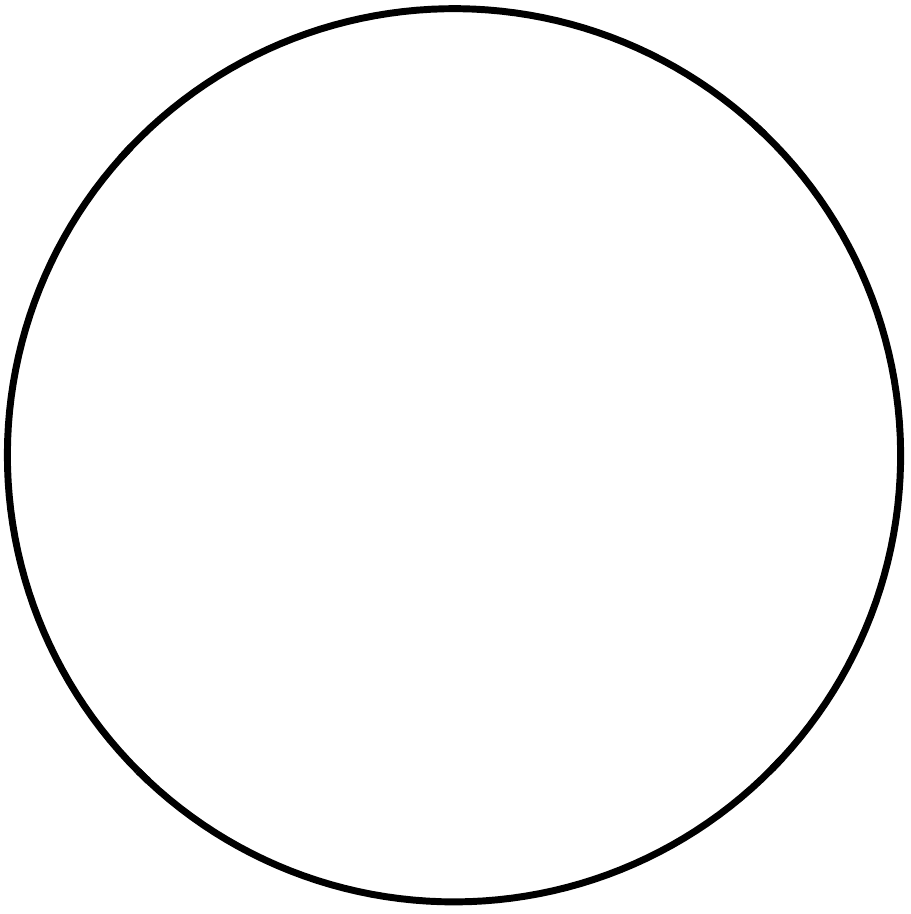}
\includegraphics[width=0.12\textwidth]{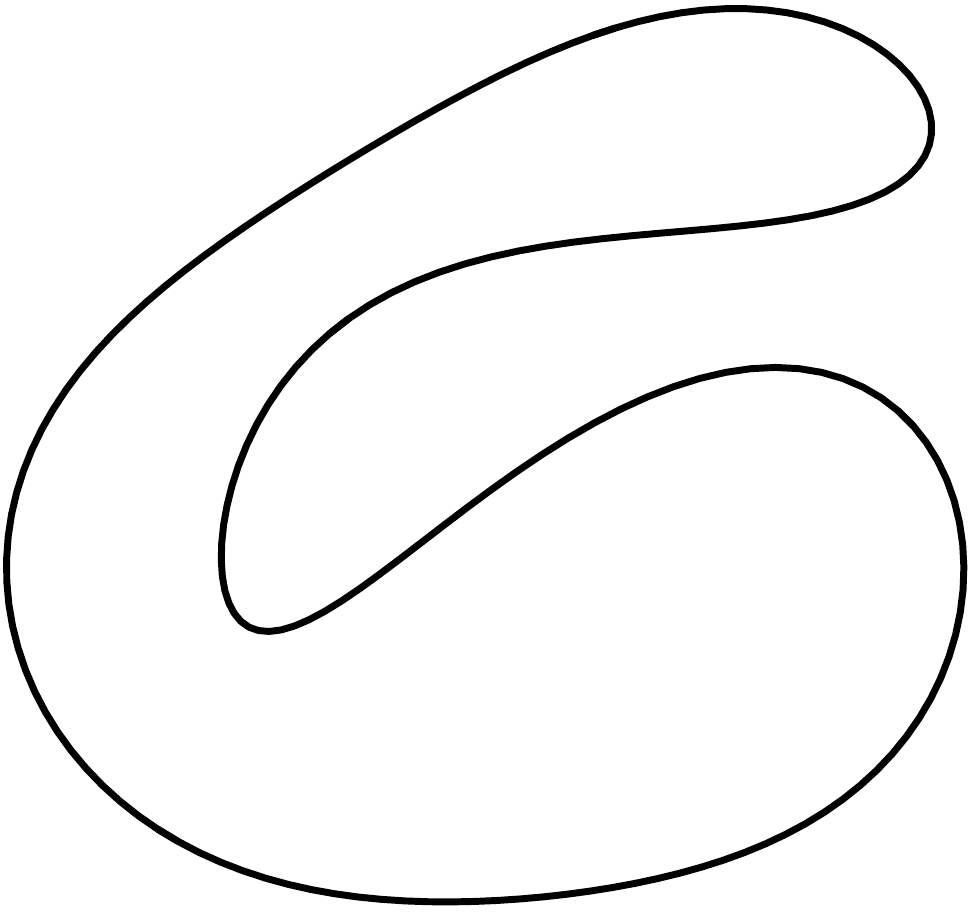}
\includegraphics[width=0.12\textwidth]{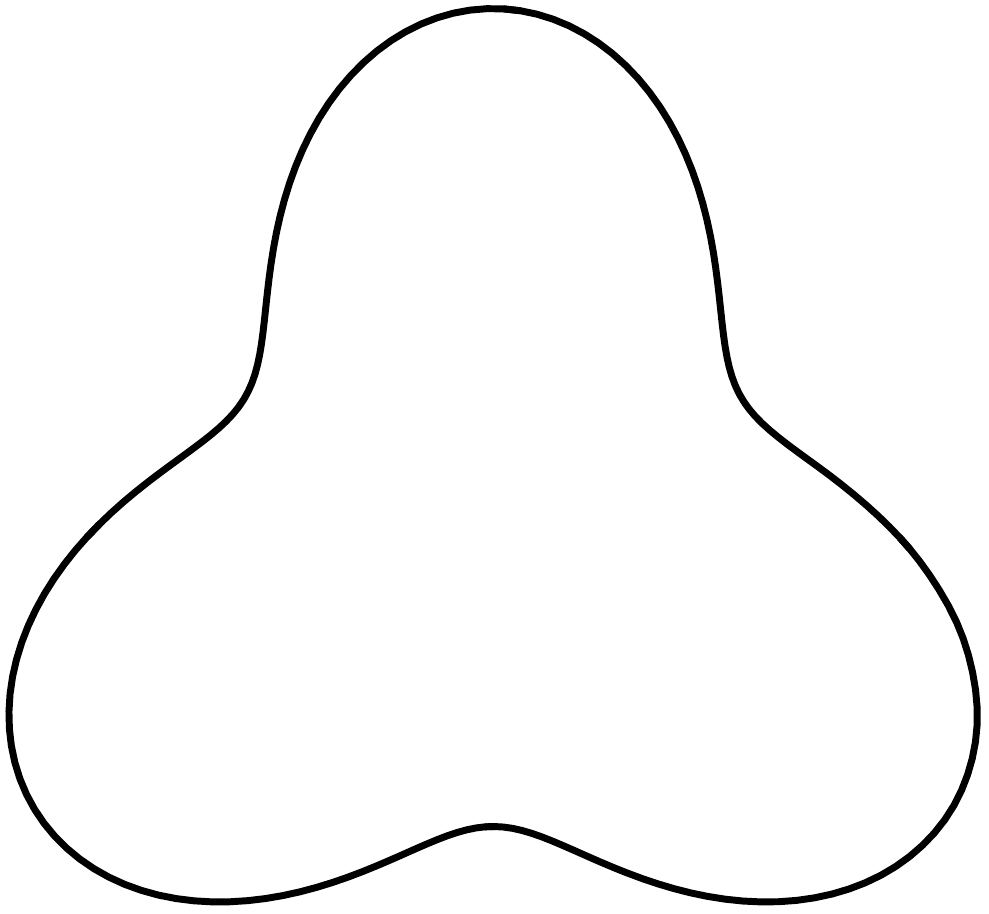}
\includegraphics[width=0.12\textwidth]{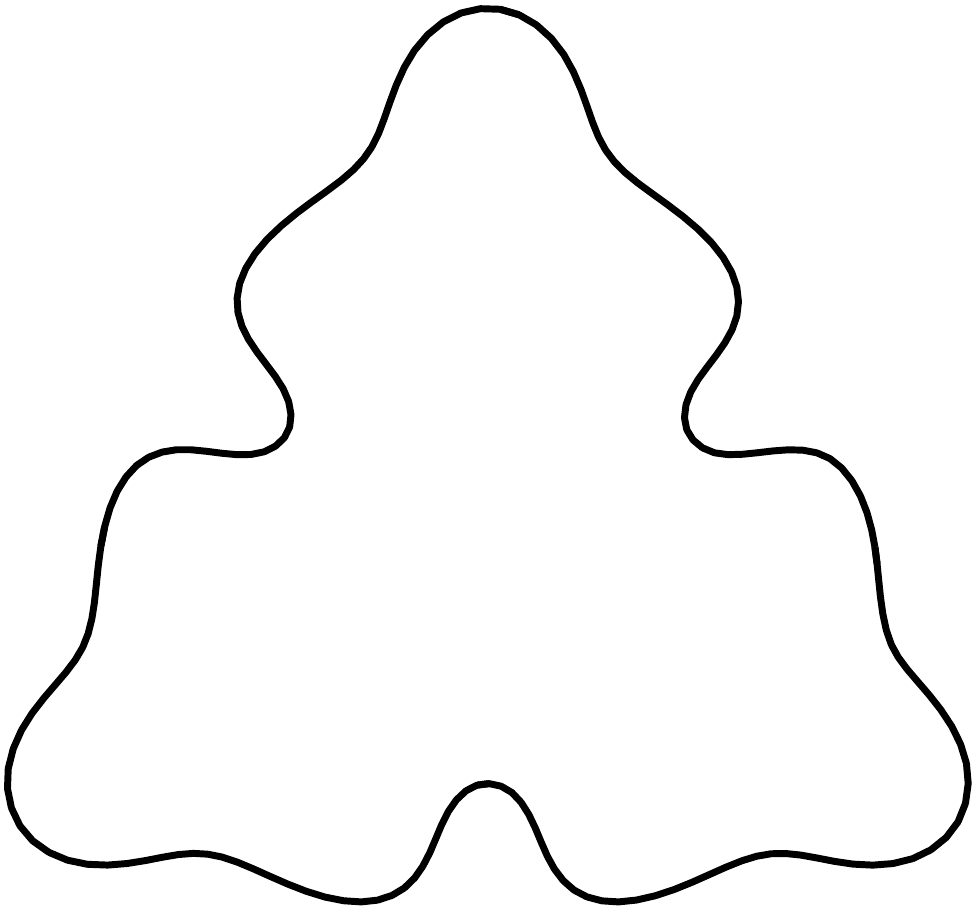}
\includegraphics[width=0.12\textwidth]{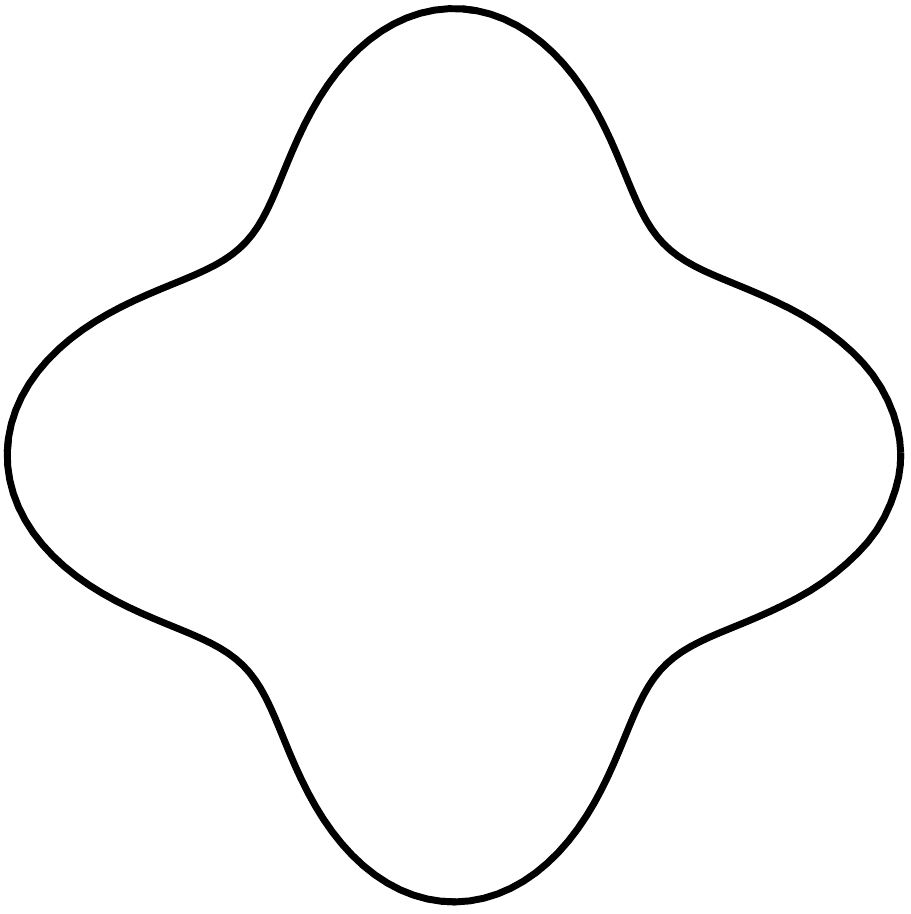}
\includegraphics[width=0.12\textwidth]{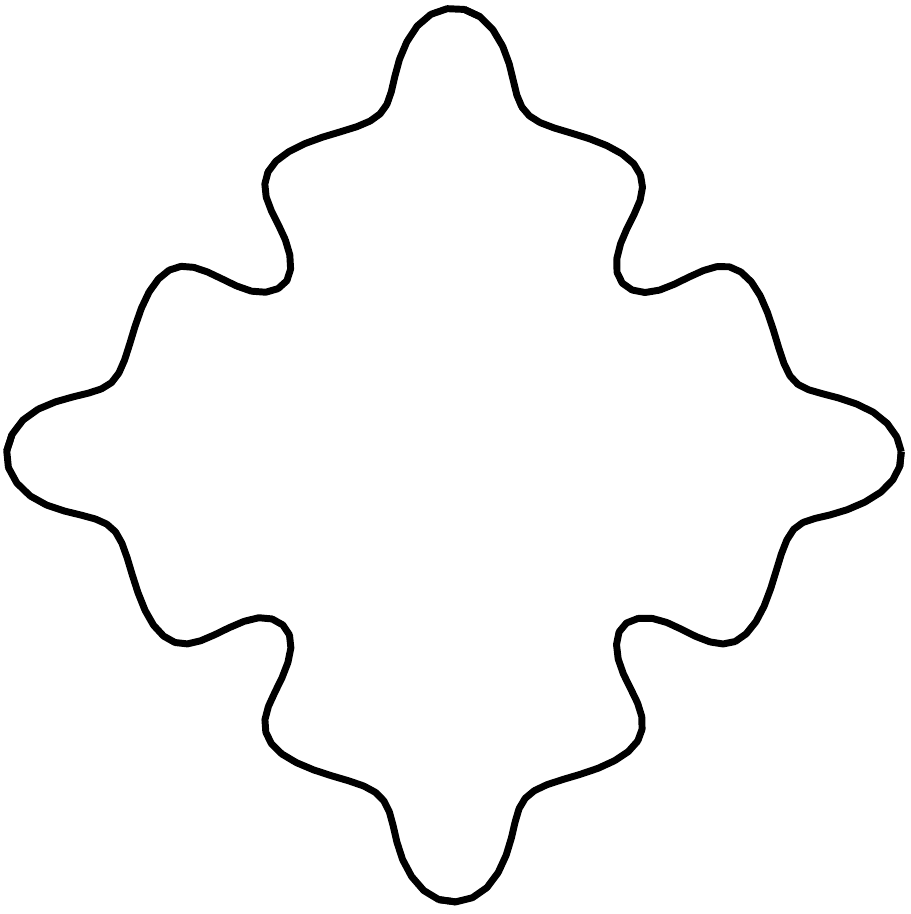}
\includegraphics[width=0.12\textwidth]{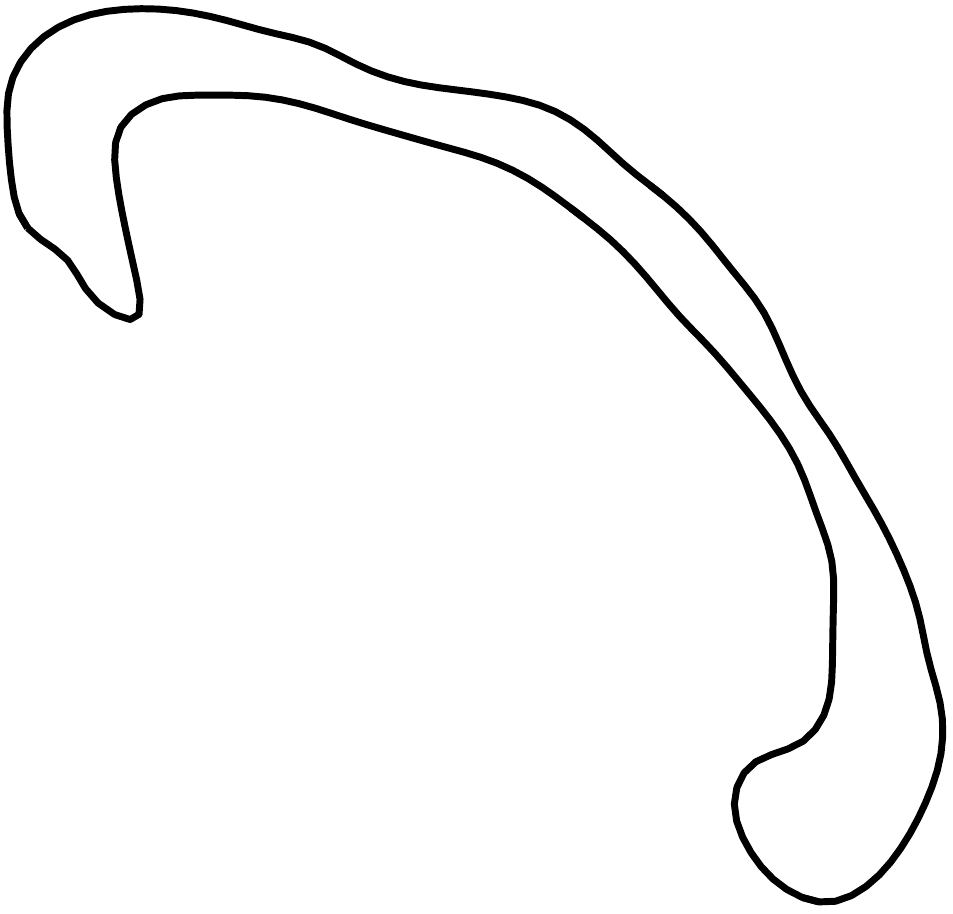}
\includegraphics[width=0.12\textwidth]{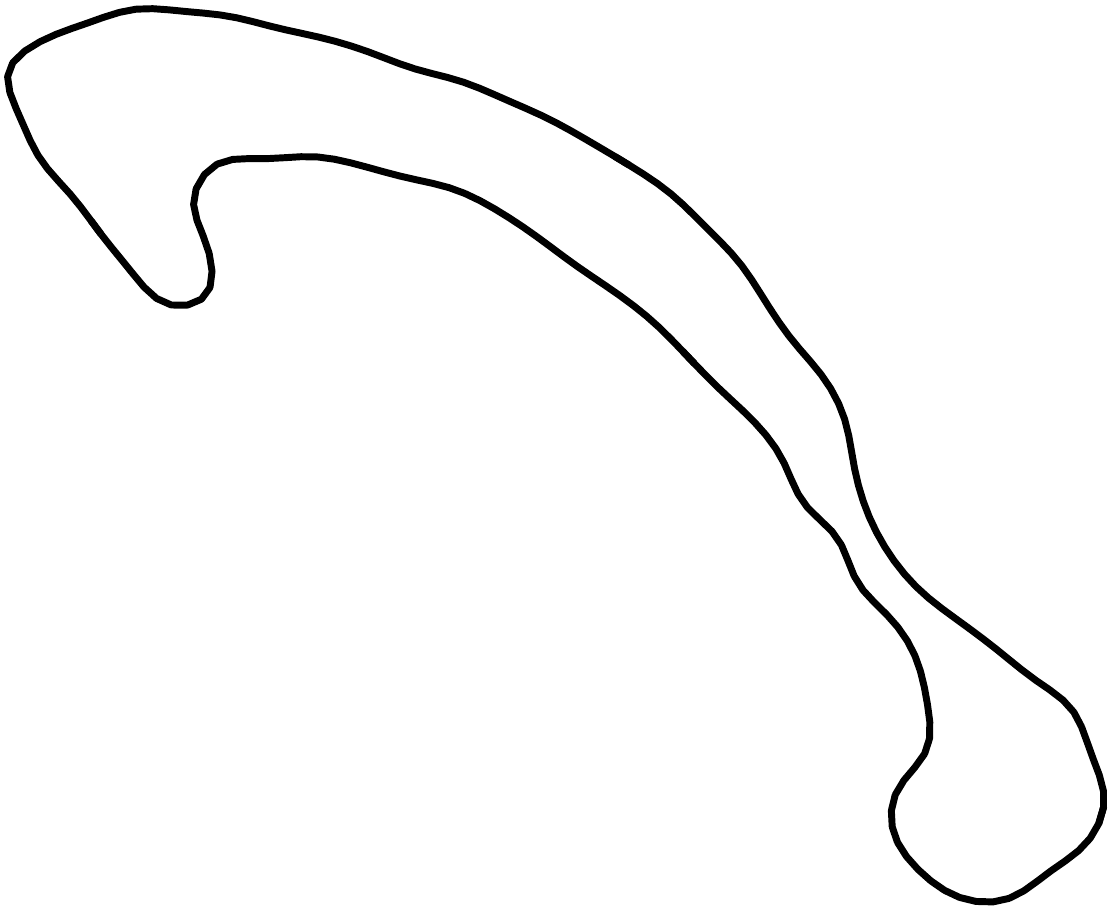}
\caption[Curves that are used in the remainder of the section to test convergence of the proposed algorithms: circle, wrap, 3- and 4-bladed propellers without and with noise and two corpus callosum shapes.]{Curves that are used in the remainder of the section to test convergence of the proposed algorithms: circle, wrap, 3- and 4-bladed propellers without and with noise, and two corpus callosum shapes.\footnotemark}
\label{fig:basic_curves}
\end{figure}
\footnotetext{The acquisition of the corpus callosum shapes is described in \cite{kucharsky2015corpus}.}

To confirm this theoretical result, we run a series of numerical experiments to test the convergence of the discrete energy, whose results are displayed in Fig.~\ref{fig:con_energy}.
The set of basic curves that we will use throughout the whole section in all numerical experiments is displayed in Fig.~\ref{fig:basic_curves}. 

\begin{figure}
\includegraphics[width=0.49\textwidth]{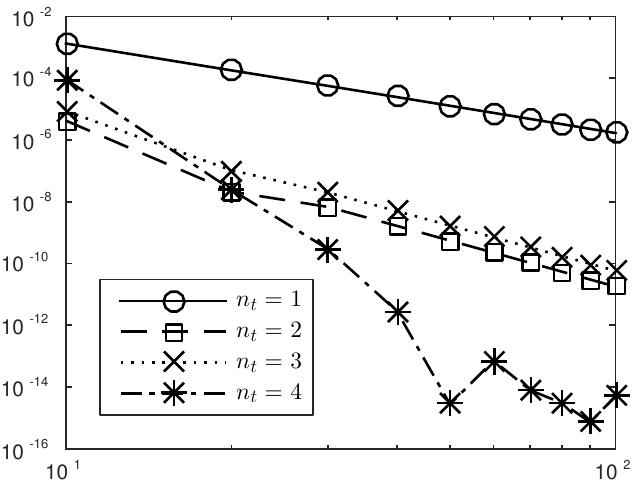}
\includegraphics[width=0.49\textwidth]{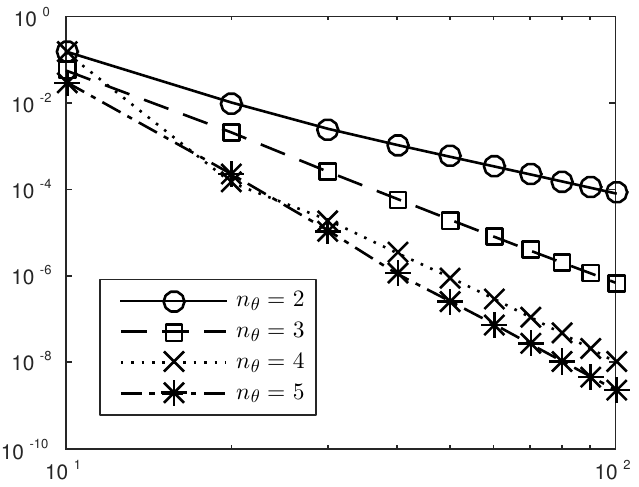}
\caption{Convergence of the discrete energy: relative energy differences for increasing number of control points of the non-linear path 
$c(t,\theta)=c_0(\theta) \sin(1 -t\pi/2)) + c_1(\theta) \sin(t\pi/2)$ connecting the circle $c_0$ to the wrap $c_1$. Left: varying $N_t$ with fixed $n_{\theta}=4,\; N_{\theta}=60$. Right: varying $N_\th$ and fixed $n_{t}=3,\; N_{t}=20$.
}
\label{fig:con_energy}
\end{figure}

\subsection{Boundary value problem for parameterized curves}

Solving the geodesic boundary problem means, for given boundary curves $c_0$ and $c_1$, to find a path $c$ which is a (local) minimum of the energy functional~\eqref{eq:EnergyFunctional} among all paths with the given boundary curves. For existence of minimizers see Theorem \ref{thm:completenessImm}. We will assume that the curves $c_0, c_1$ are discretized, i.e., given as linear combinations of the basis functions $C_j$. Should the curves be given in some other form, one would first approximate them by splines using a suitable approximation method.

The choice of full multiplicity for the boundary knots (in $t$) implies that the identity~\eqref{eq:TensorPath} for $t\in\{0,1\}$ and a spline path $c$ becomes
\begin{equation*}
c(0, \th) = \sum_{j=1}^{N_\th} c_{1,j} C_j(\th) \,, \quad c(1, \th) = \sum_{j=1}^{N_\th} c_{N_t,j} C_j(\th)\,.
\end{equation*}
If the controls $c_{1,j}$ and $c_{N_t,j}$ are fixed, then  \eqref{eq:TensorPath} defines a family of paths between between the boundary curves $c_0(\th) = \sum_{j=1}^{N_\th} c_{1,j} C_j(\th)$ and $c_1(\th) = \sum_{j=1}^{N_\th} c_{N_t,j} C_j(\th)$. The family is indexed by the remaining control points $c_{2,j},\dots,c_{N_t-1, j}$. 
Discretizing the energy functional as described in Sect.~\ref{sec:EnergyFunctional} transforms the geodesic boundary value problem to the finite-dimensional optimization problem
\begin{align}
\label{eq:unconstrainedEMinimization}
{\on{arg min}} \; E_{\on{discr}}(c_{2,1},\ldots, c_{N_{t-1},N_{\theta}})\,.
\end{align}
where $E_{\on{discr}}$ denotes the discretized energy functional $E^{m_t,m_\theta}_{N_t,N_\theta}$ applied to the spline defined by the control points $c_{i,j}$.
This finite-dimensional minimization problem can be solved by conventional black-box methods, specifically we use Matlab's \texttt{fminunc} function.
 To speed up the optimization we analytically calculated the gradient and Hessian of the energy functional $E$. We notice that
\begin{equation*}
\frac{\partial E_{\on{discr}}}{\partial c_{i,j}} = dE_c(B_i(t) C_j(\th) )\,.
\end{equation*}
The formulas for the derivative and the Hessian are provided in App.~\ref{sec:AppendixEnergyDerivatives}. 

\begin{figure}
\includegraphics[width=0.49\textwidth]{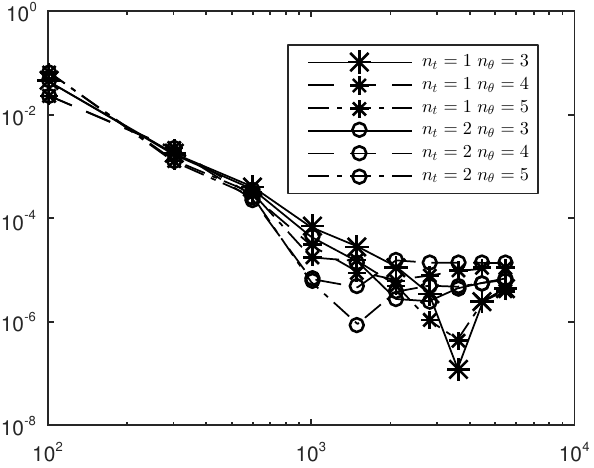}
\includegraphics[width=0.49\textwidth]{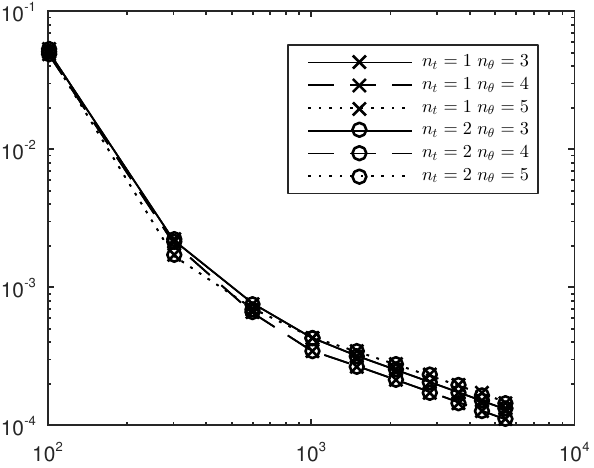}
\caption{Left/Right: Relative energy difference $\frac{|E_{i} - E_{i-1}|}{E_{i-1}}$ and $L^2$-distance $\frac{\|c_{i} - c_{i-1}\|_{L^2}}{\|c_{i-1}\|_{L^2}}$, for the propeller shapes, as a function of increasing number of control points. The values of $(N_t,N_\theta)$ are $(10,10),(15,20), \dots,(60,110)$.} 
\label{fig:geodesicdistance_conv}
\end{figure}

\emph{\begin{remark}
\emph{~For gradient-based optimization methods to work one must provide an initial path. An obvious choice for a path between two curves $c_0, c_1$ is the linear path $(1-t)c_0 + t c_1$. This path can always be constructed, but it is not always a valid initial path for the optimization procedure. For plane curves the space $\Imm(S^1,\R^2)$ is disconnected with the winding number of a curve determining the connected component \cite{Kodama2006}. The metric~\eqref{eq:sobolev_metric} is defined only for immersions, and a path leaving the space of immersion -- for example as it passes from one connected component to another -- will lead to a blow up of the energy~\eqref{eq:EnergyFunctional}. Hence an initial path connecting two curves must not leave $\on{Imm}(S^1,\R^d)$. For most examples considered in this paper the linear path is a valid initial guess; for more complicated cases a different strategy might be needed.
}\end{remark}}

\emph{\begin{remark}
\emph{~Note that the tensor product structure in \eqref{eq:TensorPath} allows us to evaluate $\on{Log}_{c_0}c_1$ by taking a time derivative of the path $c(t,\theta)$ and evaluating it at $t=0$ to obtain $\on{Log}_{c_0}c_1 = \p_t c(0,\cdot)$, where $c$ is a solution of the geodesic boundary value problem.
}\end{remark}}

\begin{figure}
\includegraphics[width=0.49\textwidth]{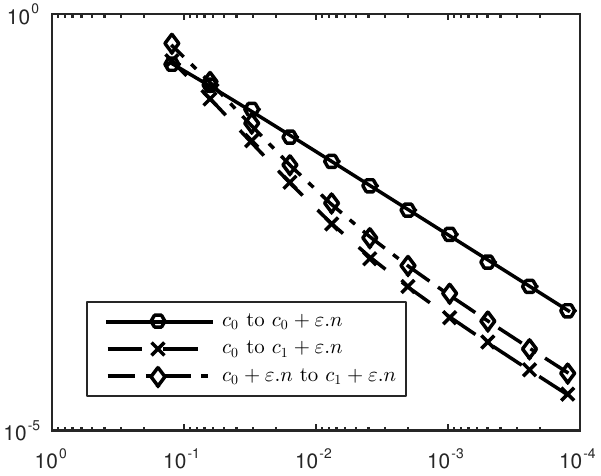}
\includegraphics[width=0.49\textwidth]{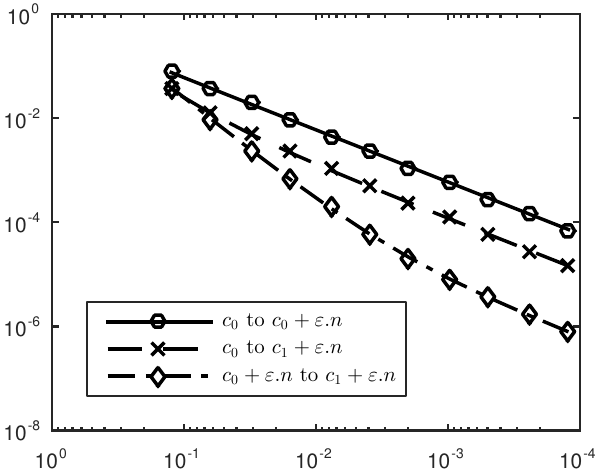}
\caption{Continuity of the geodesic distance function. Left: $c_0, c_1$ are 3- and 4-bladed propeller shapes, perturbed by a sinusoidal displacement in the normal direction of the curve. Right: $c_0, c_1$ are corpus callosum shapes with the perturbation applied directly to the control points. The plots show the relative change in distance against the amplitude of the sinusoidal noise $\ep$, i.e., $\on{dist}(c_0, c_1 + \ep.n) / \on{dist}(c_0, c_1)$.
}
\label{fig:continuity}
\end{figure}

Now we prove a result about $\Gamma$-convergence of the discrete energy functionals. Before stating the theorem, we set up some notation.
For brevity we denote the spline approximation operators in Lemma \ref{lem:tensorspline} by $\mathcal S_N$ with $N = (N_t,N_\th)$, the space $H^{1,2}([0,1]\times S^1;\mathbb R^d)$ by $H^{1,2}$, and the space $H^2(S^1;\mathbb R^d)$ by $H^2$. Fix $c_0, c_1 \in H^2$ and let $\Omega_{c_0,c_1}H^{1,2}$ denote the weakly closed subset of paths $c \in H^{1,2}$ with $c(0,\cdot) = c_0$ and $c(1,\cdot) = c_1$. We define the following discrete energy functionals on the set $U = \left\{ c \in H^{1,2} \,:\, c'(t,\th) \neq 0\,\forall (t,\th) \right\} $:
\begin{equation*}
E_{N,\Omega}(c) = \begin{cases}
E(c)\,,& c \in \mathcal S_N\left(\Omega_{c_0, c_1} H^{1,2}\right) \\
\infty\,,& \text{otherwise}
\end{cases}\,.
\end{equation*}
In other words, $E_{N,\Om}(c)$ equals $E(c)$, if $c$ is a spline path connecting the spline approximations $\mathcal S_{N_\th}(c_0)$, $\mathcal S_{N_\th}(c_1)$ of $c_0$, $c_1$, and equals $\infty$ otherwise. We will show that the limit of these functionals as $(N_t,N_\th) \to \infty$ is
\begin{equation*}
E_\Omega(c) = \begin{cases}
E(c)\,,& c \in \Omega_{c_0, c_1} H^{1,2} \\
\infty\,,& \text{otherwise}
\end{cases}\,.
\end{equation*}
Our result is the following.
\begin{theorem}
Let $n_t\geq 1$ and $n_\theta\geq 2$. Then the discretized energy functionals $E_{N,\Omega}$ are equi-coercive on $U$ with respect to the weak $H^{1,2}$-topology and $\Gamma$-converge on $U$ with respect to the weak $H^{1,2}$-topology to the energy functional $E_\Omega$ as $N = (N_t,N_\theta)\to\infty$. It follows that every sequence of minimizers of the discretized energy functionals $E_{N,\Omega}$ has a subsequence that converges weakly to a minimizer of $E_\Omega$. 
\end{theorem}

We refer to \cite[Definitions~4.1 and 7.6]{Maso1993} for the concepts of equi-coercivity and $\Gamma$-convergence and to \cite[Chapter~8]{Maso1993} for $\Gamma$-convergence under weak topologies. 

\begin{proof}
First we show that the functionals $E_{N,\Omega}$ are equi-coercive with respect to the weak topology on $H^{1,2}$. This means that for each $r>0$ there is a weakly compact set $K_r$ such that for each $N$, $\{c \in U: E_{N,\Omega}(c)\leq r^2 \}\subseteq K_r$. To see this let $r>0$ and $c \in U$ with $E_{N,\Omega}(c)\leq r^2 $. Then $E(c)\leq r^2 $ and consequently $L(c) \leq r$. Therefore, $\on{dist}(c(0),c(t))\leq r$ for all $t \in [0,1]$. By \cite[Prop.~3.5 and Lem.~4.2]{Bruveris2014b_preprint} there exist constants $C_1,C_2>0$ such that $\|h\|_{H^2}^2\leq C_1 G_{\tilde c}(h,h)$ and $\|c(0) - \tilde{c}\|_{H^2} \leq C_2 \on{dist}(c(0),\tilde{c})$ for all $\tilde c \in \mathcal I^2(S^1,\mathbb R^d)$ satisfying $\on{dist}(c(0),\tilde c)\leq r$ and all $h \in H^2$. Since we have full multiplicity at the ends we have $c(0) = S_{N_\theta}^{n_\theta}(c_0)$, by classical approximation results there exists a constant $C_3$ such that $\| S_{N_\theta}^{n_\theta}(c_0) \|_{H^2} \leq C_3 \|c_0\|_{H^2}$, where $C_3$ does not depend on $N_\theta$. This allows us to estimate
\begin{align*}
\|c\|_{H^{1,2}}^2 
&=
\int_0^1 \| c(t) \|_{H^2}^2 + \|\dot c(t)\|^2_{H^2} \ud t \\
&\leq
\| c(0) \|^2_{H^2} + C^2_2 r^2 + C_1 E(c) \\
&\leq
C^2_3\, \| c_0 \|^2_{H^2} + \left(C_2^2 + C_1\right) r^2 =:R_r^2\,.
\end{align*} 
This shows that $\{c \in U:E_{N,\Omega}(c)\leq r^2 \}$ is contained in the set $K_r$, defined as the closed ball of radius $R_r$ around the origin in $H^{1,2}$. Since closed balls in Hilbert spaces are weakly compact, the functionals $E_{N,\Omega}$ are equi-coercive.

Equi-coercivity allows us to apply \cite[Prop.~8.16]{Maso1993}, giving the following sequential characterization of $\Gamma$-convergence: the functionals $E_{N,\Omega}$ $\Gamma$-converge to $E_\Omega$ with respect to the weak $H^{1,2}$-topology, if 
\begin{align}\label{equ:gamma1}
\forall c\ \forall c_N &\rightharpoonup c: & E_\Omega(c)&\leq \liminf_{N\to\infty}E_{N,\Omega}(c_N)\,, 
\\\label{equ:gamma2}
\forall c\ \exists c_N &\rightharpoonup c: & E_\Omega(c) &= \lim_{N\to\infty}E_{N,\Omega}(c_N)\,.
\end{align}
To prove \eqref{equ:gamma1} let $c_N \rightharpoonup c$ in $H^{1,2}$. First we consider the case $c \in \Omega_{c_0,c_1}H^{1,2}$. Then $E(c) = E_\Omega(c)$. Notice that $E(c) \leq E_{N,\Omega}(c)$ always. It is shown in the proof of \cite[Thm~5.2]{Bruveris2014b_preprint} that $E$ is sequentially weakly lower semicontinuous, so we obtain
\begin{equation*}
E_\Omega(c) = E(c) \leq \liminf_{n \to \infty} E(c_N) \leq \liminf_{n \to \infty} E_{N,\Omega}(c_N)\,.
\end{equation*}
Now if $c \notin \Omega_{c_0,c_1}H^{1,2}$ then $c_N \notin \Omega_{c_0,c_1}H^{1,2}$ for almost all $N$, because the latter set is weakly closed. Thus we have both $E_{\Om}(c) = E_{N,\Om}(c) = \infty$ and \eqref{equ:gamma1} is satisfied.
Equation \eqref{equ:gamma2} follows directly from Lem. \ref{lem:continuity} by choosing the sequence $c_N = S^{n_\th}_{N_\th}(c)$. Thus, we have shown that the functionals $E_{N,\Omega}$ are equi-coercive and $\Gamma$-converge to $E_\Omega$ in the weak $H^{1,2}$ topology. The statement about the convergence of minimizers follows by applying \cite[Thm.~7.23]{Maso1993}.
\end{proof}

Note that the result concerns the energy functional restricted to spline spaces, but evaluation of integrals is assumed to be exact. For the case where we approximate the integrals by gaussian quadrature, we were not able to prove a $\Gamma$-convergence result. However, in numerical experiments we still observe convergence for the solution of the boundary value problem, as can be seen in Fig.~\ref{fig:geodesicdistance_conv} for varying numbers of control points. Convergence holds for both the optimal energy and the $L^2$-norm of the minimizing paths. In Fig.~\ref{fig:continuity} we show that the geodesic distance function is continuous: by adding a sinusoidal displacement in the normal direction to the curves, the geodesic distance converges to 0 as the noise becomes smaller. 

\subsection{Boundary value problem for unparameterized curves}
To numerically solve the boundary value problem on the space of unparametrized curves, we first have to discretize the diffeomorphism group.  
By the identification of $S^1$ with $\mathbb{R}/[0,2\pi]$, diffeomorphisms $\ph\colon S^1 \to S^1$ can be written as $\ph = \operatorname{Id} + f$, where  $f$ is a periodic function. Periodic functions can be discretized as before using simple knot sequences with periodic boundary conditions. This leads to the spline representation
\[
\ph(\theta) = \sum_{i=1}^{N_{\ph}} \ph_i D_i(\theta)=\sum_{i=1}^{N_{\ph}} \left(\xi_i + f_i\right) D_i(\theta)\,.
\]
Here $D_i$ are B-splines of degree $n_\ph$, defined on a uniform periodic knot sequence, 
$f_i$ are the control points of $f$, i.e., $f(\theta)=\sum_{i=1}^{N_\ph} f_i D_i(\theta)$, and 
$\xi_i$ are the \emph{Greville abscissas}, i.e., control points of the identity represented in a B-spline basis, $\operatorname{Id} = \sum_{i=1}^{N_\ph} \xi_i D_i$. 

\begin{figure}
\includegraphics[width=0.49\textwidth]{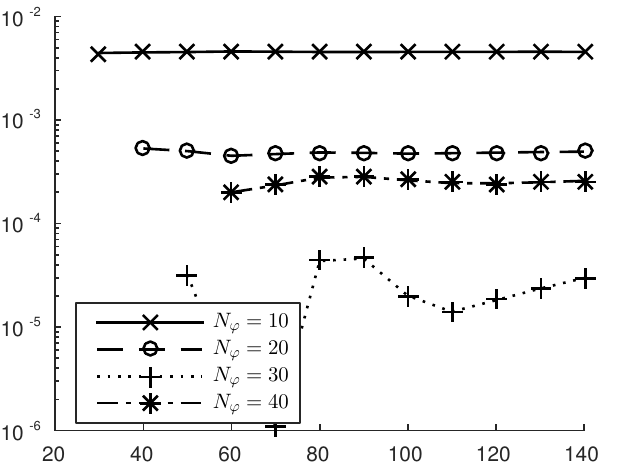}
\includegraphics[width=0.49\textwidth]{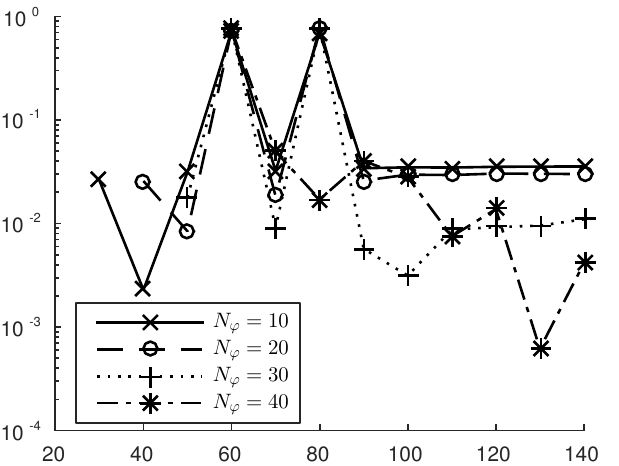}
\caption{Symmetry of the geodesic distances for the 3- and 4-bladed propeller shapes on the left and the corpus callosum shapes on the right. 
The relative difference $|\on{dist}(c_0, c_1) - \on{dist}(c_1, c_0)| / \on{dist}(c_0, c_1)$ is plotted against $N_\th$ for different choices of $N_\ph$.
}
\label{fig:symmetry}
\end{figure}

The constraint that $\varphi$ is a diffeomorphism is $\varphi' > 0$. By the fact that the B-spline basis functions are nonnegative and by the recursive formula for the derivatives of B-splines, see  \cite[Chap.~4]{schmeisser2007recent}, a sufficient condition to ensure that $\varphi' > 0$ is
\begin{equation}
f_{i-1} - f_i < \xi_{i} - \xi_{i-1}\,.
\label{eq:PhiDiffeoCondition}
\end{equation}  
This is a linear inequality constraint. To speed up convergence, we introduce an additional variable $\alpha \in \mathbb R$ representing constant shifts of the reparametrization. The resulting redundancy is eliminated by the constraint
\begin{equation}\label{eq: shift constraint}
	\sum_{i=1}^{N_\ph} f_i = 0\,,
\end{equation}
which ensures that the average shift of $\ph$ is 0.

We have to minimize the energy functional \eqref{eq:EnergyFunctional} over all paths $c\colon[0,1] \times [0,2\pi] \to \mathbb{R}^2$, and diffeomorphisms $\ph$, subject to the constraints
\begin{equation*}
	c(0, \cdot) = c_0\,, \qquad
	c(1, \cdot) = c_1 \circ \ph\,.
\end{equation*}
It is important to note that the reparametrization $(c, \ph) \mapsto c \circ \ph$ does not preserve splines: if $c_1$ and $\ph$ are represented by splines, then the function $c_1 \circ \ph$ is in general not. To overcome this difficulty 
we approximate the reparameterized curve $c_1 \circ \ph$ by a new spline in each optimization step. This then leads to a 
finite-dimensional constrained minimization problem
\begin{align}
\label{eq:constrainedEMinimization}
\on{argmin} E_{\on{discr}}(c_{2,1},\ldots,c_{N_{t}-1,N_{\theta}},f_1,\ldots,f_{N_\ph},\alpha)\,,
\end{align}
where $f_1,\ldots,f_{N_\ph}$ are the controls used to construct the diffeomorphism $\ph$ and $\alpha$ is the constant shift in the parametrization. Similar to the unconstrained problem \eqref{eq:unconstrainedEMinimization}, we can analytically compute the gradient and hessian and then solve this by standard methods for constrained minimization problems, specifically we use Matlab's \texttt{fmincon} function. In order to use \texttt{fmincon} we replace \eqref{eq:PhiDiffeoCondition} by $f_{i-1} - f_i \leq \xi_i - \xi_{i-1} + \ep$ with $\ep$ small.

From a mathematical point of view we would expect the geodesic distance between two shapes to be symmetric, i.e., interchanging the curves $c_0$ and $c_1$ should have no effect on the resulting geodesic distance. This is, however, only approximately true numerically: in our numerical examples the relative error is below 5\% if sufficiently many grid points are chosen, see Fig.~\ref{fig:symmetry}. We want to point out that Fig.~\ref{fig:symmetry} does not demonstrate convergence of the relative error to zero. Indeed, our numerical experiments do not confirm this. This problem has not received sufficient attention in some of the previous literature; in particular, \cite{Jermyn2011,Vialard2014_preprint,Esl2014} seem to sidestep this question.
An example of a forward and backward geodesic is plotted in Fig.~\ref{fig:symmetry_bvp}.

\subsection{Boundary value problem on shape space}
To numerically solve the boundary value problem on shape space, it remains to discretize the finite-dimensional motion group.  
To simplify the presentation we will assume in the following that $d=2$, so that we can parametrize rotations by the one-dimensional parameter $\beta$. We have to add translations and rotations to the minimization problem, i.e., minimize the energy functional \eqref{eq:EnergyFunctional} over all paths $c\colon[0,1] \times [0,2\pi] \to \mathbb{R}^2$, diffeomorphisms $\ph$, rotations $R_{\beta}$ and translations $a$, subject to the constraints
\begin{equation*}
	c(0, \cdot) = c_0\,, \qquad
	c(1, \cdot) = R_{\beta}(c_1 \circ \ph+a)\,.
\end{equation*}
Note that rotations and translations preserve splines: if $c_1$ is represented by a spline then for any rotation $R_{\beta}$ and translations $a$ the function $R_{\beta}(c_1+a)$ is 
a spline of the same type. 
This then leads to a 
finite-dimensional constrained minimization problem
\begin{align}
\on{argmin} E_{\on{discr}}(c_{2,1},\ldots,c_{N_{t}-1,N_{\theta}},f_1,\ldots,f_{N_\ph},\alpha,\beta,a)\,,
\end{align}
where $f_1,\ldots,f_{N_\ph}$ are the controls used to construct the diffeomorphism $\ph$, $\alpha$ is the constant shift in the parametrization, $\beta$ the rotation angle and $a$ the translation vector.

\begin{figure}
\centering
\includegraphics[width=0.80\textwidth]{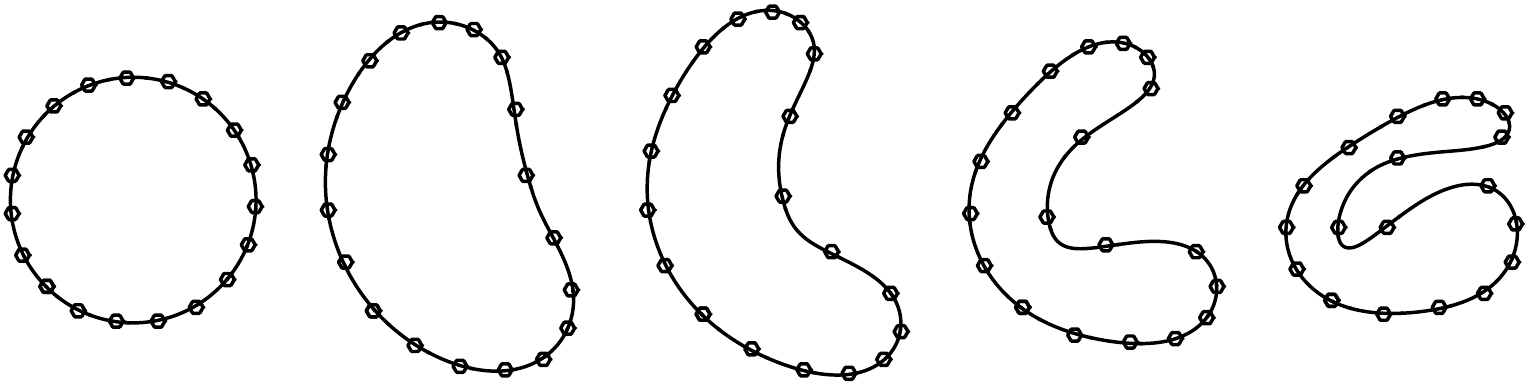}\hspace{0.02\textwidth}
\includegraphics[width=0.80\textwidth]{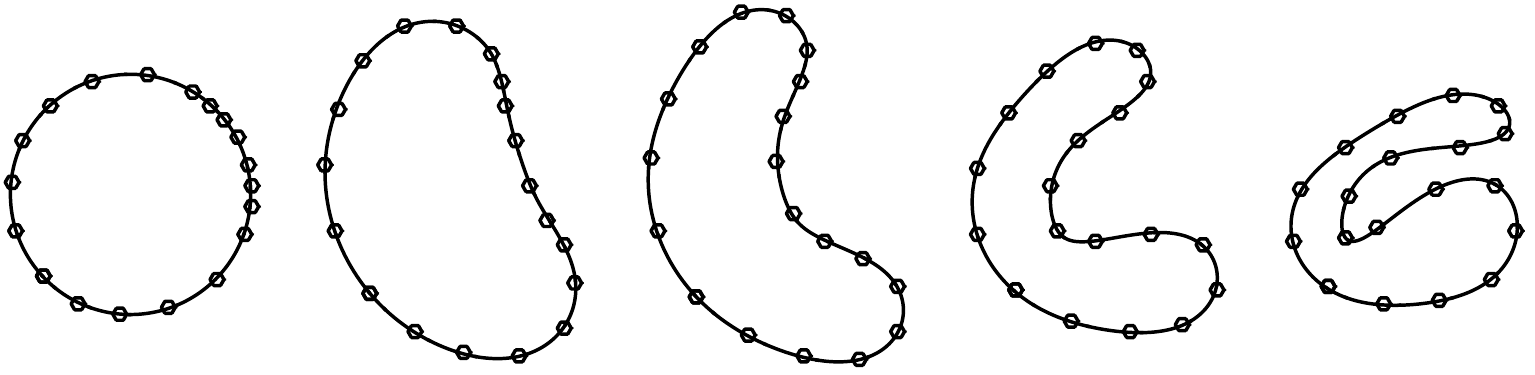}\hspace{0.02\textwidth}
\caption{Symmetry of the geodesic boundary value problem on the space of unparametrized curves. For the circle and the wrap the geodesic boundary value problem is solved forwards and backwards. 
To better compare the results the second geodesic is plotted backwards in time. The plot markers visualize the optimal parametrization of the curves.}
\label{fig:symmetry_bvp}
\end{figure}

\subsection{Initial value problem}

To solve the geodesic initial value problem we use the variational discrete geodesic calculus developed in \cite{Rumpf2014}. For a discrete path $(c_0, \dots, c_K)$, $K \in \mathbb N$, one defines the discrete energy 
\[
E_K(c_0, \dots, c_K) = K \sum_{k=1}^K W(c_{k-1}, c_k)\,,
\]
where $W(c, \tilde c)$ is an approximation of $\on{dist}(c, \tilde c)^2$. Since our Riemannian metric $G$ is smooth, it approximates the squared distance sufficiently well in the sense that $G_c(c-\tilde c,c-\tilde c)-\on{dist}(c,\tilde c)^2=O(\on{dist}(c,\tilde c)^3)$, and we can take the approximation to be
\[
W(c, \tilde c) = \frac 12 G_{c}(c - \tilde c , c - \tilde c)\,.
\]
We call $(c_0, \dots, c_K)$ a discrete geodesic if it is a minimizer of the discrete energy with fixed endpoints $c_0, c_K$. To define the discrete exponential map we consider discrete paths $(c_0, c_1, c_2)$ consisting of three points. The discrete energy of such a path is 
\[
E_2(c_0, c_1, c_2) = G_{c_0}(c_1 - c_0, c_1 - c_0) + G_{c_1}(c_2 - c_1, c_2 - c_1)\,.
\]
Given $c_0, c_1$, we define $c_2 = \on{Exp}_{c_0}c_1$ if $(c_0, c_1, c_2)$ is a discrete geodesic, i.e., if $c_2$ is such  that $c_1 = \on{argmin} E_2(c_0, \cdot, c_2)$. Given an initial curve $c_0$, an initial velocity $v_0$, and a number $K$ of time steps, our solution of the geodesic initial value problem is $c_K = \on{Exp}_{c_{K-2}}c_{K-1}$, where the intermediate points $c_1, \dots, c_{K-1}$ are defined iteratively via
\[
c_1 = c_0 + \frac 1K v_0\,,\quad
c_2 = \on{Exp}_{c_0}c_1\,,\quad
c_3 = \on{Exp}_{c_1}c_2\,,\quad\dots\,,\quad
c_{K-1} = \on{Exp}_{c_{K-3}}c_{K-2}\,.
\]

To compute a discrete geodesic we need to find minima of the function $E_2(c_0,\cdot,c_2)$, that is, we compute $c_2$ such that $c_1$ is a minimizer of $E_2(c_0,\cdot,c_2)$. Differentiating $E_2$ with respect to $c_1$ leads to the following system of nonlinear equations
\[
2G_{c_0}(c_1 - c_0, \cdot) - 2G_{c_1}(c_2 - c_1, \cdot) + D_{c_1}G_{\cdot}(c_2 - c_1, c_2 - c_1) = 0\,.
\]
This system has to be solved for $c_2$, with the argument replaced by all basis functions $C_j$ defining the spline space. We use the solver \texttt{fsolve} in Matlab to solve this system of equations. Some examples of discrete geodesics are depicted in Fig.~\ref{fig:hela_5pc}. The discretizations of the geodesic initial and boundary value problems are compatible as demonstrated in Fig.~\ref{fig:compatible}.

\begin{figure}
\includegraphics[width=0.49\textwidth]{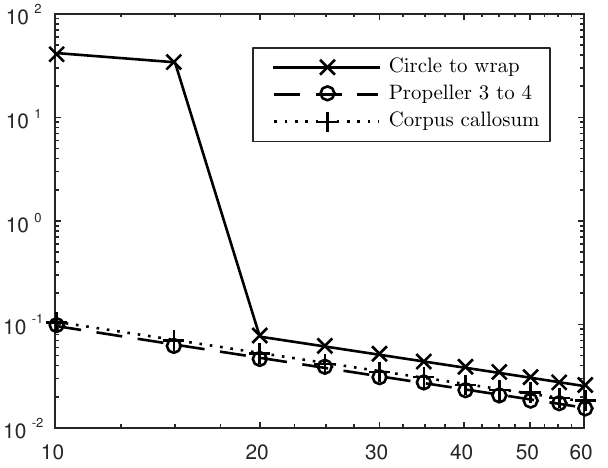}
\includegraphics[width=0.49\textwidth]{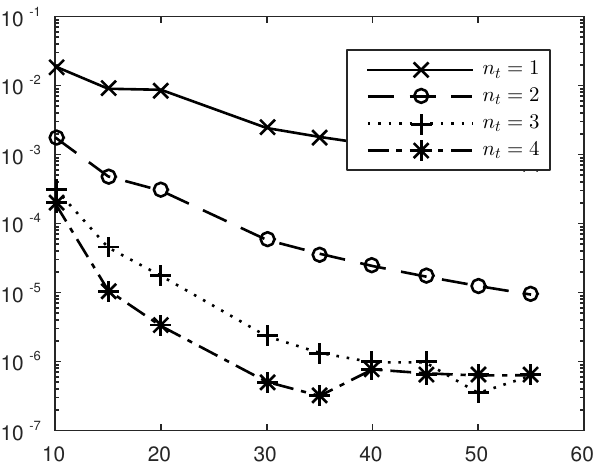}
\caption{Compatibility of the geodesic IVP and BVP with increasing $N_t$. On the left one computes $c_1 = \on{Exp}_{c_0}(v)$ for given $c_0, v$ and then solves the BVP for $\tilde v = \on{Log}_{c_0}(c_1)$. The plot shows the relative distance $\| v - \tilde v \|_{c_0} / \| v \|_{c_0}$ agains $N_t$. On the right one computes $v = \on{Log}_{c_0}(c_1)$ for given $c_0, c_1$ and plots the relative difference $\| v - \tilde v\|_{c_0} / \on{dist}(c_0,c_1)$ between two consecutive (w.r.t. $N_t$) initial velocities $v, \tilde v$ against $N_t$.
}\label{fig:compatible}
\end{figure}

\subsection{Karcher mean}

The Karcher mean $\overline{c}$ of a set $\{c_1,\dots,c_n\}$ of curves is the minimizer of
\begin{equation}
\label{eq: KarcherEnergy}
F(c) = \frac{1}{n} \sum_{j=1}^n \on{dist}(c, c_j)^2\,.
\end{equation}
It can be calculated by a gradient descent on $(\on{Imm}(S^1,\R^d), G)$. Letting $\on{Log}_{c}c_j$ denote the Riemannian logarithm, the gradient of $F$ with respect to $G$ is \cite{Pennec2006b}
\begin{equation}
\on{grad}^GF (c) = -\frac 2n \sum_{j=1}^n \on{Log}_c c_j\,.
\end{equation}
Fig.~\ref{fig:propellerMean} illustrates the computation of the Karcher mean of 8 propeller shapes, which have all been modified by adding a 10\% uniform noise to their control points. 
\begin{figure}
\centering
\includegraphics[width=0.99\textwidth]{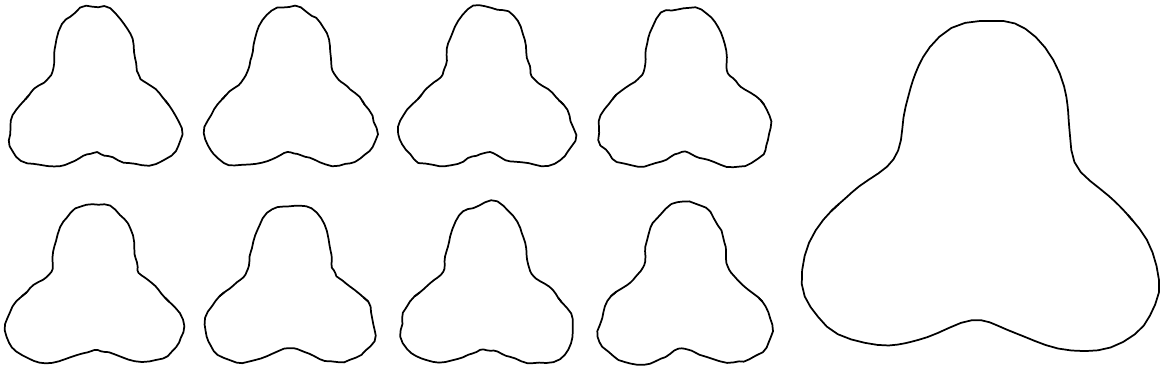}
\caption{Eight propellers with 10\% uniform noise added to their control points, along with their Karcher mean.}
\label{fig:propellerMean}
\end{figure}

\section{Shape analysis of HeLa cells}

We used second order metrics to characterize the nuclear shape variation in HeLa cells. The data consists of fluorescence microscope images of HeLa cell nuclei\footnote{The dataset was downloaded from \url{http://murphylab.web.cmu.edu/data}.} (87 images in total). The acquisition of the images is described in \cite{Boland2001}.

To extract the boundary of the nucleus we apply a thresholding method \cite{Otsu1979} to obtain a binary image, and then we fit -- using least squares -- a spline with $N_\th=12$ and $n_\th=4$ to the longest $4$-connected component of the thresholded image. This provides a good balance between capturing shape details and not overfitting the image noise; see Fig.~\ref{fig:hela_images}. Then we rescale all curves by the same factor to arrive at an average length $\bar\ell_c = 2\pi$. The choice $\bar\ell_c = 2\pi$ has the following nice property: if a curve $c$ has $\ell_c = 2\pi$ and $c$ has a constant speed parametrization, then $|c'| = 1$, and the arc length derivative $D_s h$ coincides with the regular derivative $h'$. The scaling matters because the metric we work with is not scale invariant. Had we decided to work with curves of a different average length we would have to change the constants $a_j$ of the metric in order to arrive at the same results. 

\begin{figure}
\centering
\includegraphics[width=0.8\textwidth]{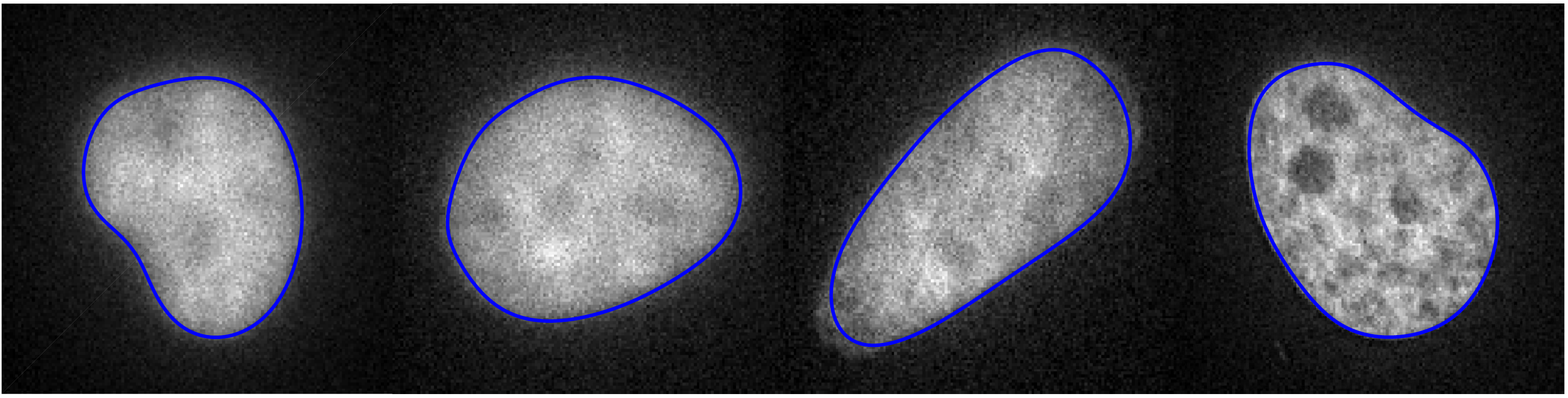}
\caption{Examples of HeLa cell nuclei and the spline representation of the boundary.}
\label{fig:hela_images}
\end{figure}

For the subsequent analysis we use splines with $N_\th=40$ and $n_\th=3$. The increased number of control points compared to the data acquisition allows us to preserve shape information even after reparametrizing the curves. To parametrize the diffeomorphism group we use splines with $N_\ph=20$ and $n_\ph=3$. This leaves us with roughly $2\cdot 40 - 20 - 2 - 1 = 57$ degrees of freedom to represent the population of $87$ given shapes of cell nuclei. The influence of the number of control points on the geodesic BVP can be seen in Fig.~\ref{fig:hela_Nth}. All analysis is performed modulo translations, rotations, and reparametrizations. 

\begin{figure}
\centering
\includegraphics[width=0.6\textwidth]{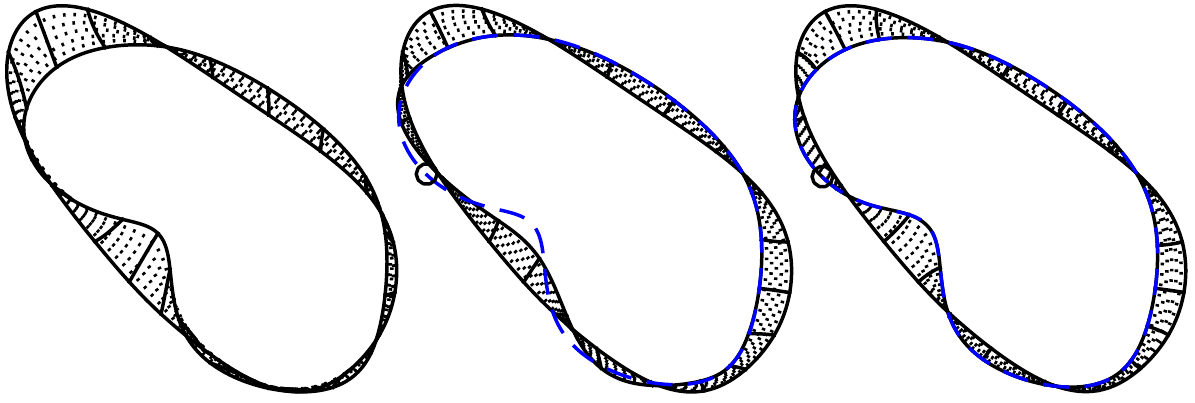}
\caption{Geodesic between two cells (solid lines); the dashed line shows the exact endpoint before reparametrization. The geodesic is computed between parametrized curves with $N_\th=12$ (left), unparametrized curves with $N_\th=12$ (middle) and $N_\th=40$ (right).}
\label{fig:hela_Nth}
\end{figure}

The choice of constants $a_0$, $a_1$, and $a_2$ of the Riemannian metric has a significant impact on the results; see Fig~\ref{fig:var_expl}. One constant may be chosen freely, so we set $a_0=1$. To simplify the interpretation of the results, we set $a_1=0$; our metric shall have no $H^1$-part. This leaves us with one more parameter, which we choose by looking at the $L^2$- and $H^2$-contributions to the energy of geodesics between shapes in the dataset. For a geodesic $c$ between two curves $c_0$ and $c_1$ these contributions are
\[
\on{dist}(c_0, c_1)^2 = E_{L^2}(c) + E_{H^2}(c) 
= \int_0^1 \int_{S^1} |c_t|^2 \ud s \ud t 
+ a_2 \int_0 \int_{S^1} \left| D_s^2 c_t\right|^2 \ud s \ud t\,.
\]
The relative contribution of the $H^2$-term to the total energy is $\rh_{H^2} = E_{H^2} / (E_{L^2} + E_{H^2})$. We denote the population mean and standard deviation of the variable $\rh_{H^2}$ by $\bar\rh_{H^2}$ and $\si$, respectively. The following table shows that the choices $a_2=2^{-12} \approx 0.00024$ and $a_2 = 2^{-8} \approx 0.0039$ both lead to balanced energy contributions of the zero and second order terms:
\begin{align*}
a_2 &= 2^{-12}\,, &
\bar\rh_{H^2} & = 0.032 \,,&
\si &= 0.027 \,,\\
a_2 &= 2^{-8} \,,&
\bar\rh_{H^2} & = 0.203 \,,&
\si &= 0.119 \,.
\end{align*}
We will use these parameter choices in our subsequent analysis. Note that from a physical point of view the parameter $a_2$ has units $[m^4]$, $m$ being meters.

The average shape of the nucleus can be captured by the Karcher mean $\bar c$. To solve the minimization problem \eqref{eq: KarcherEnergy} for the Karcher mean of the 87 nuclei we 
use a conjugate gradient method on the Riemannian manifold of curves as implemented in the Manopt library \cite{Manopt2014}. For each choice of parameters the optimization is performed until the gradient of the objective function $F(\bar c)$ satisfies $\|\on{grad}^G F(\bar c)\|_{\bar c} < 10^{-3}$.

Having computed the mean $\bar c$, we represent each nuclear shape $c_j$ by the initial velocity $v_j = \on{Log}_{\bar c}(c_j)$ of the minimal geodesic from $\bar c$ to $c_j$. We perform principal component analysis with respect to the inner product $G_{\bar c}$ on the set of initial velocities $\{v_j \,:\, j = 1,\dots,87\}$. Geodesics from the mean in the first five directions can be seen in Fig.~\ref{fig:hela_5pc}. A projection of the dataset onto the subspace spanned by the first two principal components is depicted in Fig.~\ref{fig:hela_2d}.

\begin{figure}
\centering
\includegraphics[width=0.8\textwidth]{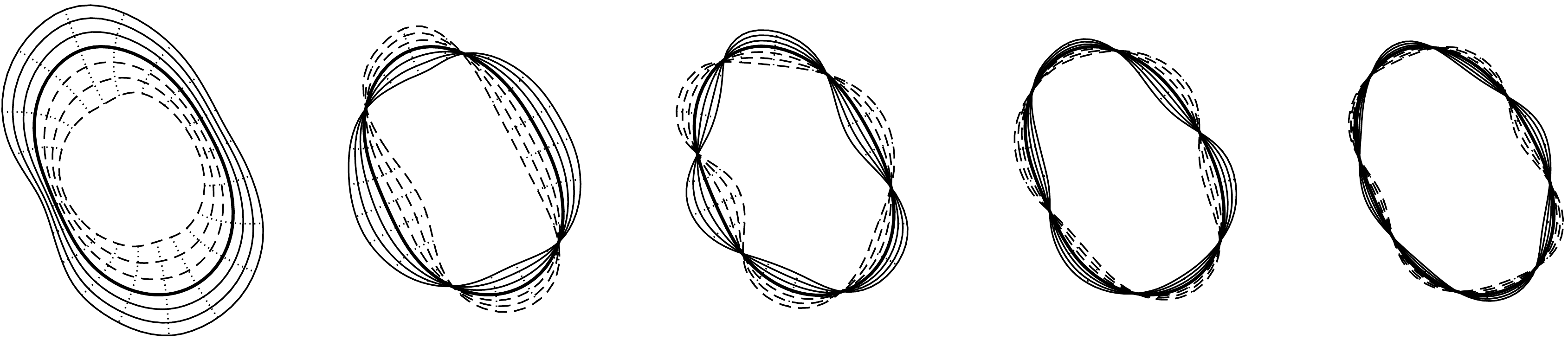} \\
\includegraphics[width=0.8\textwidth]{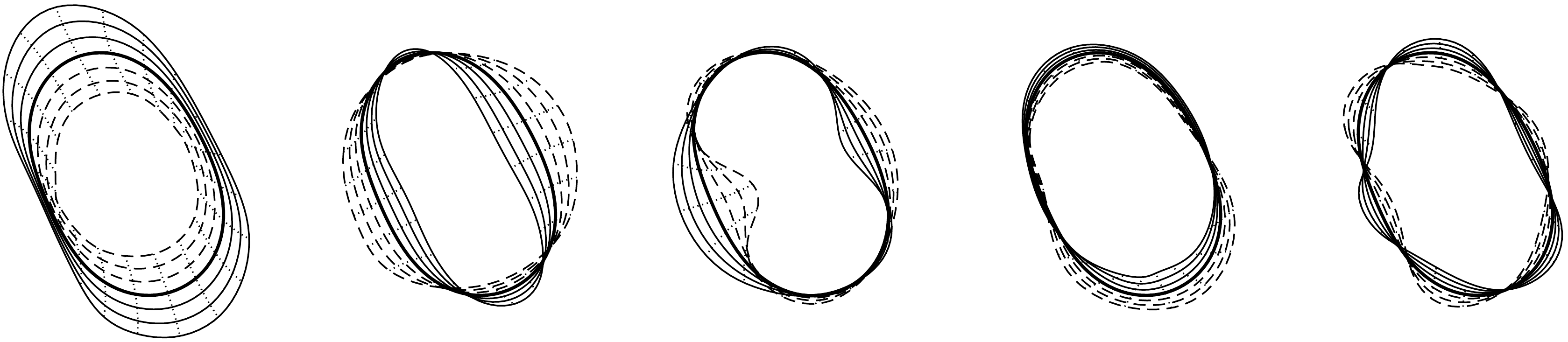}
\caption{Geodesics from the mean in the first five principal directions. The curves show geodesics at times $-3, -2, \dots, 2, 3$; the mean is shown in bold. One can see characteristic deformations of the curve: expansion, stretching, compressing and bending. The first row shows principal components calculated for curves modulo reparametrizations; the second row for parametrized curves.}
\label{fig:hela_5pc}
\end{figure}

\begin{figure}
\centering
\includegraphics[width=0.6\textwidth]{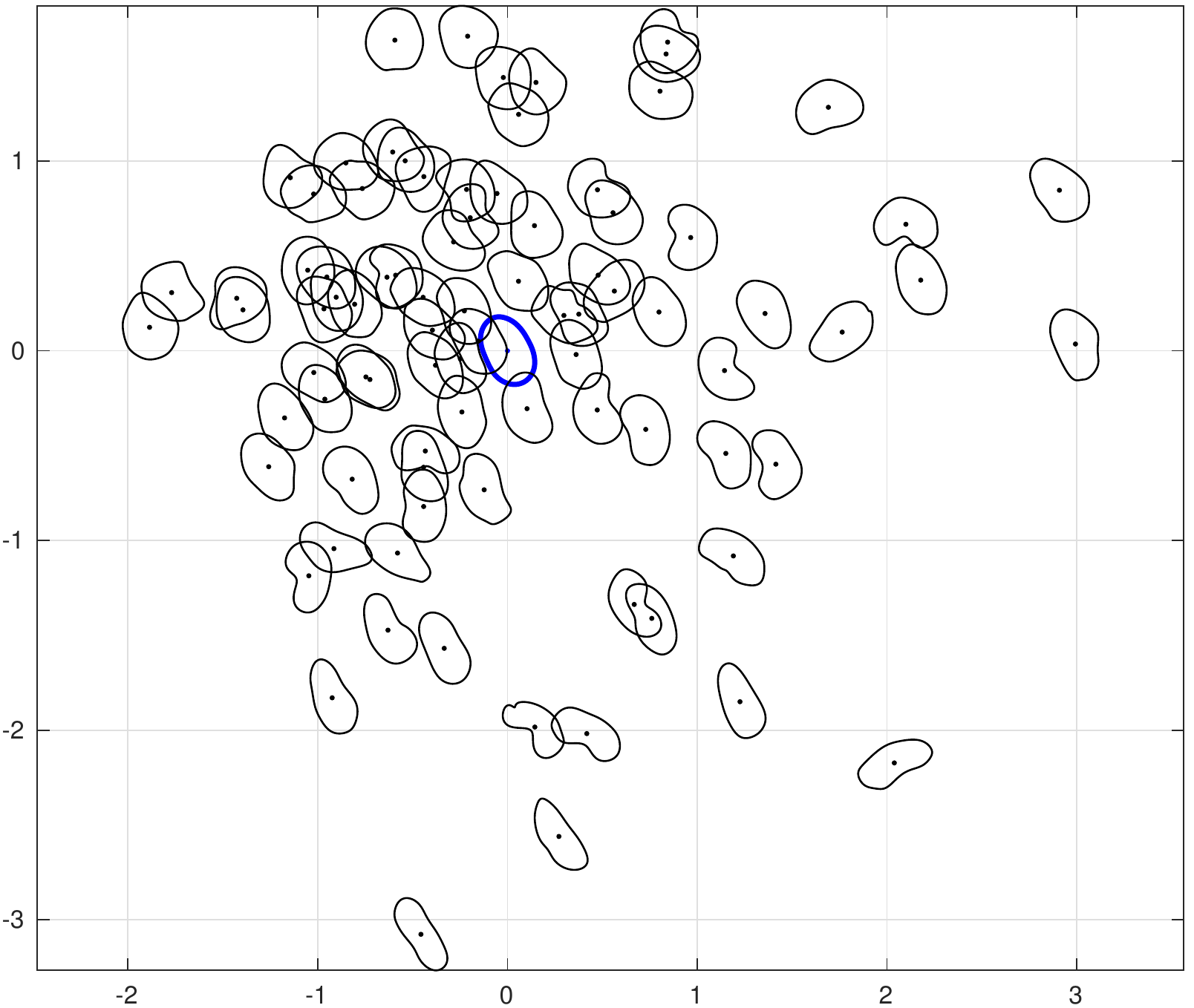}
\caption{Cell nuclei projected to the plane in the tangent plane, spanned by the first two principal components. The mean (in blue) is situated at the origin. The units on the coordinate axes are standard deviations.}
\label{fig:hela_2d}
\end{figure}

For unparametrized curves and for the parameter choice $a_2=2^{-12}$ the first five principal components explain $57.6\%$, $78.3\%$, $90.0\%$, $94.2\%$ and $98.0\%$ of the total variance; see Fig.~\ref{fig:var_expl}. Under the choice  $a_2=2^{-8}$ the first five principal components explain only $93.3\%$ of the variance as compared to $98.0\%$ in the previous case. This demonstrates that approximation power of the principal components depends on the choice of the metric. Fig.~\ref{fig:var_expl} also shows that fewer principal components are needed to explain the same amount of variance when the reparametrization group is factored out.

\begin{figure}
\centering
\raisebox{-.5\height}{\includegraphics[width=0.18\textwidth]{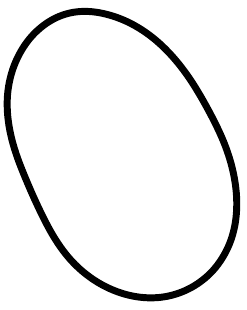}}
\hspace{0.01\textwidth}
\raisebox{-.5\height}{\includegraphics[width=0.15\textwidth]{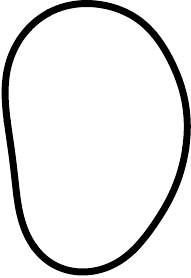}}
\hspace{0.02\textwidth}
\raisebox{-.5\height}{\includegraphics[width=0.49\textwidth]{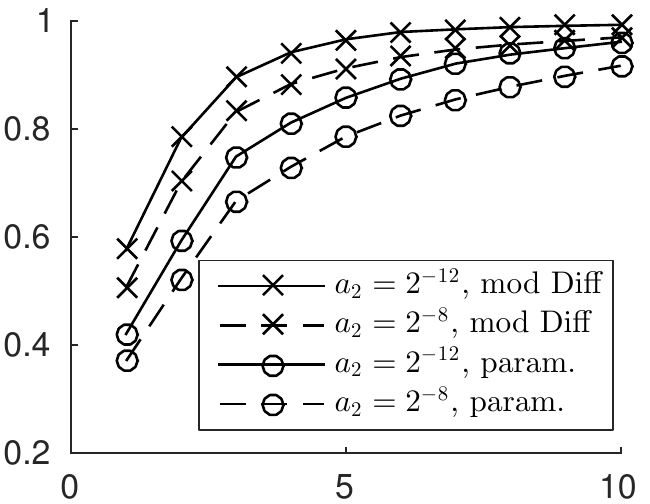}}
\caption{Left: the mean shape of cell nuclei. Middle: for comparison,  the mean shape as computed in \cite{Rohde2008} via the Christensen--Rabbitt--Miller method \cite{Christensen1996}. Right: the proportion of the total variance explained by the first 10 eigenvectors.}
\label{fig:var_expl}
\end{figure}

The results we obtain are comparable to those of \cite{Rohde2008}, where diffeomorphic matching was used to compare cells. It turns out that the mean shape with respect to our metrics is symmetric, while the mean shape obtained in \cite{Rohde2008} is bent towards one side; see Fig.~\ref{fig:var_expl}.

\appendix
\section{Convergence of spline approximations}
\label{app:convergence}

The Hilbert space tensor product $H^k([0,1])\widehat\otimes H^\ell(S^1)$ is the completion of the algebraic tensor product $H^k([0,1])\otimes H^\ell(S^1)$ with respect to the uniform cross norm 
\begin{align}
\beta\bigg(\sum_i f_i\otimes g_i\bigg)^2 = \sum_{i,j} \langle f_i,f_j\rangle_{H^k([0,1])} \langle g_i,g_j\rangle_{H^\ell(S^1)}\,.
\end{align}
The following result connects the mixed order Sobolev space \eqref{Eq:mixedorderSob} to a Hilbert space tensor product. 

\begin{lemma}\label{lem:tensor}
$H^{k,\ell}([0,1]\times S^1)$ is isometrically isomorphic to $H^k([0,1])\widehat\otimes H^\ell(S^1)$.
\end{lemma}

A similar result for $H^{k,\ell}(\R\times \R)$ is shown in \cite[Thm.~2.1]{sickel2009tensor}. Our proof follows the lines of \cite[Thm.~1.39]{light1985approximation}, where the result is shown for the case $k=\ell=0$.

\begin{proof}
Each tensor $c=\sum_i f_i\otimes g_i \in H^k([0,1]) \otimes H^\ell(S^1)$ defines a function $Jc \in H^{k,\ell}([0,1]\times S^1)$, 
via  $Jc(t_1,t_2) = \sum_i f_i(t_1) g_i(t_2)$. It is not hard to verify that $J$ is an isometric 
embedding of $H^k([0,1]) \otimes H^\ell(S^1)$ in $H^{k,\ell}([0,1]\times S^1)$, i.e., $\beta(c) = \|c\|_{H^{k,\ell}([0,1]\times S^1)}$. 
To complete the proof, we show that $J$ is onto. Being an isometry, the range of $J$ is closed and so it suffices to show that its orthogonal complement is trivial in $H^{k,\ell}([0,1]\times S^1)$. Let $c \in H^{k,\ell}([0,1]\times S^1)$ 
and suppose that $\langle c, f\otimes g\rangle_{H^{k,\ell}([0,1]\times S^1)}=0$ for all $f \in H^k([0,1])$ and $g\in H^\ell(S^1)$. Let $\langle c,g\rangle_{H^\ell(S^1)}$ 
denote the function $t_1\mapsto \int_{S^1} c(t_1,t_2)g(t_2)dt_2$. Then $\langle c,g\rangle_{H^\ell(S^1)} \in H^k(S^1)$ with $\partial_{t_1}^k \langle c,g\rangle_{H^\ell(S^1)} = \langle \partial_{t_1}^k c,g\rangle_{H^\ell(S^1)}$. It follows that
\begin{align*}
\langle c,f\otimes g\rangle_{H^{k,\ell}([0,1]\times S^1)}
=
\langle f,\langle c, g\rangle_{H^\ell(S^1)}\rangle_{H^k([0,1])} = 0\,.
\end{align*}
As $f$ is arbitrary, it follows that $\langle c,g\rangle_{H^\ell(S^1)}$ 
vanishes at almost every $t_1$. Similarly, since $g$ is arbitrary, $c$ vanishes  at almost every $t_1, t_2$.
Therefore, $c=0$ in $H^{k,\ell}([0,1]\times S^1)$.
\end{proof}

\begin{corollary}\label{cor:simpletensors}
The multiplicatively decomposable functions $(t, \theta) \mapsto f(t)g(\theta) = (f \otimes g)(t, \theta)$ with $f \in H^k([0,1])$, $g\in H^\ell(S^1)$, span a dense subspace of $H^{k,\ell}([0,1]\times S^1)$.  
\end{corollary}
\begin{proof}
This follows from the denseness of $H^k([0,1])\otimes H^\ell(S^1)$ in $H^k([0,1])
\operatorname{\widehat \otimes} H^\ell(S^1)$ and Lem.~\ref{lem:tensor}.
\end{proof}

The following lemma shows that the Sobolev embedding theorem in one dimension extends to higher dimensions via tensor products. 

\begin{lemma}\label{lem:embedding}
For each $k,\ell\geq 0$, the space $H^{k+1,\ell+1}([0,1]\times S^1)$ is continuously embedded in the space $C^{k,\ell}([0,1]\times S^1)$.
\end{lemma}

\begin{proof}
Let $\{f_i\}$ be an orthonormal basis of $H^{k+1}([0,1])$ and $\{g_j\}$ an orthonormal basis of $H^{\ell+1}(S^1)$. Then $\{f_i\otimes g_j\}$ is an orthonormal basis of $H^{k+1}([0,1])\widehat\otimes H^{\ell+1}(S^1)$. By Lem.~\ref{lem:tensor} this space is equal to $H^{k+1,\ell+1}([0,1]\times S^1)$. Therefore, any element in this space can be expressed as $c=\sum_{i,j} c_{i,j} f_i\otimes g_j$ for some $c_{i,j}\in\mathbb R$. By the Sobolev embedding theorem in one dimension there exists $C>0$ such that
\begin{align*}
\|\partial_t^{k}\partial_\theta^l c\|_\infty^2
&= 
\left\|\sum_{ij}c_{ij}(\partial_t^k f_i)\otimes(\partial_\theta^l g_j)\right\|_\infty^2
\leq 
\sum_{ij} c_{ij}^2 \|\partial_\theta^k f_i\|^2_\infty \|\partial_t^l g_j\|^2_\infty
\\&\leq 
C \sum_{ij} c_{ij}^2 \|f_i\|_{H^{k+1}([0,1])}^2 \|g_j\|_{H^{\ell+1}(S^1)}^2
= C \|c\|_{H^{k+1,\ell+1}([0,1]\times S^1)}^2\,.
\end{align*}
Similar estimates hold for lower derivatives of $c$. This shows that the $C^{k,\ell}$-norm is bounded by the $H^{k+1,\ell+1}$-norm. 
\end{proof}

To prove Lem.~\ref{lem:tensorspline} we need a result on the approximation power of one-dimensional splines.

\begin{lemma}\label{lem:onesplines}
Let $I=[0,1]$ or $I=S^1$, $n,k \in \mathbb N$ with $n \geq k$, and $f \in H^{k}(I)$. Then
\begin{equation*}
\lim_{N\to\infty} \| f - \mathcal S^n_{N}f \|_{H^{k}(I)} = 0\,.  
\end{equation*}
\end{lemma} 

\begin{proof}
The set of smooth functions in dense in $H^k(I)$. Therefore, there is for each $\ep >0$ a function $g \in C^\infty(I)$ such that $\|f-g\|_{H^k(I)}<\ep/2$. If $N$ is sufficiently large, there is a spline $h$ of order $n$ defined on the uniform grid with $N$ points such that $\|g-h\|_{H^k(I)}<\ep/2$. This follows from \cite[Cor.~6.26]{Schumaker2007}. By the best approximation property of the orthogonal projection $\mathcal S^n_N$,  
\begin{equation*}
\|f-\mathcal S^n_Nf\|_{H^k(I)}
\leq \|f-h\|_{H^k(I)} 
\leq \|f-g\|_{H^k(I)} +\|g-h\|_{H^k(I)} <\ep\,.
\end{equation*}
Since $\ep$ was arbitrary, this shows that $\mathcal S^n_N f \to f$ in $H^n(I)$ as $N\to \infty$.
\end{proof} 	

Collecting these results we are able to prove Lem.~\ref{lem:tensorspline}.

\color{black}{\mbox{\color{header1}{\it Proof of Lem.~\ref{lem:tensorspline}}.$\,$}\color{black}}
Let $c\in H^{k,\ell}([0,1]\times S^1)$ and $\ep > 0$. By Cor.~\ref{cor:simpletensors} there exist functions $f_i\in H^{k}([0,1])$ and $g_i \in H^{\ell}(S^1)$, $i=1,\dots,n$, such that 
\begin{equation*}
\left\|c-\sum_{i=1}^{n}f_i\otimes g_i\right\|_{H^{k,\ell}([0,1]\times S^1)}<\ep/2\,.
\end{equation*}
By Lem.~\ref{lem:onesplines} and by the fact that the tensor norm is a reasonable cross norm (i.e., $\| f_i\otimes g_i \|_{H^{k,\ell}([0,1]\times S^1)}\leq \| f_i \|_{H^{k}([0,1])}\| g_i \|_{H^{\ell}(S^1)}$) it is possible to choose $N_t$ and $N_\theta$ large enough such that
\begin{equation*}
\left\|\sum_{i=1}^n f_i\otimes g_i-\sum_{i=1}^n \mathcal S^{n_t}_{N_t}f_i \otimes \mathcal S^{n_\theta}_{N_\theta}g_i\right\|_{H^{k,\ell}([0,1]\times S^1)}<\ep/2\,.
\end{equation*}
These two estimates and the best approximation property of the orthogonal projection $\mathcal S^{n_t,n_\theta}_{N_t,N_\theta}$ yield
\begin{equation*}
\left\|c-\mathcal S^{n_t,n_\theta}_{N_t,N_\theta}c\right\|_{H^{k,\ell}([0,1]\times S^1)}
\leq 
\left\|c-\sum_{i=1}^n \mathcal S^{n_t}_{N_t}f_i \otimes \mathcal S^{n_\theta}_{N_\theta}g_i\right\|_{H^{k,\ell}([0,1]\times S^1)} 
<\ep\,.\qquad{\color{header1}\rule{1.5ex}{1.5ex}}
\end{equation*}

\section{Derivatives of the energy functional}
In this appendix we list the derivatives of the energy 
functional~\eqref{eq:EnergyFunctional}. The first derivative is
\label{sec:AppendixEnergyDerivatives}
\begin{align*}
dE_c(k) = \int_0^1 \int_0^{2\pi} &t_1 \ip{c'}{k'} + t_2 \left( \ip{c''}{k'} + \ip{c'}{k''} \right)  + t_3 \ip{\dot{c}}{\dot{k}} + t_4 \ip{\dot{c}'}{\dot{k}'} \\
&  + t_5 ( \ip{\dot{c}''}{\dot{k}'} + \ip{\dot{c}'}{\dot{k}''} ) + t_6 \ip{\dot{c}''}{\dot{k}''} 
 \ud \theta \ud t\,,
\end{align*}
with
\begin{align*}
t_1 &=  \frac{a_0}{|c'|} \ip{\dot{c}}{\dot{c}} -
\frac{a_1}{|c'|^3} \ip{\dot{c}'}{\dot{c}'}
 - 7\frac{a_2}{|c'|^9} \ip{c'}{c''}^2 \ip{\dot{c}'}{\dot{c}'}
 + 10\frac{a_2}{|c'|^7} \ip{c'}{c''} \ip{\dot{c}'}{\dot{c}''}
 - 3\frac{a_2}{|c'|^5} \ip{\dot{c}''}{\dot{c}''} \,,\\
t_2 &=  2\frac{a_2}{|c'|^7} \ip{c'}{c''} \ip{\dot{c}'}{\dot{c}'}
 - 2\frac{a_2}{|c'|^5} \ip{\dot{c}'}{\dot{c}''}\,, \quad
t_3 = 2a_0 |c'| \,, \quad
t_4 = 2\frac{a_1}{|c'|} + 2\frac{a_2}{|c'|^7} \ip{c'}{c''} \,,\\
t_5 &= - 2\frac{a_2}{|c'|^5} \ip{c'}{c''} \,, \quad
t_6 = 2\frac{a_2}{|c'|^3}\,. 
\end{align*}
The Hessian is
\begin{align*}
d^2 E_c(h,k) &= \int_0^1 \int_0^{2\pi} w_1 \ip{c'}{h'} \ip{c'}{k'} \\
& + w_2 \left( \ip{c''}{h'}\ip{c'}{k'} + \ip{c'}{h'}\ip{c''}{k'} + \ip{c'}{h''}\ip{c'}{k'} + \ip{c'}{k''}\ip{c'}{h'} \right) \\
& + w_3 ( \ip{c''}{h'}\ip{c''}{k'} + \ip{c'}{h''}\ip{c'}{k''} + \ip{c'}{h''}\ip{c''}{k'} + \ip{c'}{k''}\ip{c''}{h'} )  \\
& + w_4 ( \ip{\dot{c}}{\dot{h}}\ip{c'}{k'} + \ip{\dot{c}}{\dot{k}}\ip{c'}{h'}  )
 + w_5 ( \ip{\dot{c}'}{\dot{h}'}\ip{c'}{k'} + \ip{\dot{c}'}{\dot{k}'}\ip{c'}{h'} ) \\
& + w_6 ( \ip{\dot{c}''}{\dot{h}'}\ip{c'}{k'} + \ip{\dot{c}''}{\dot{k}'}\ip{c'}{h'} + \ip{\dot{c}'}{\dot{h}''}\ip{c'}{k'} + \ip{\dot{c}'}{\dot{k}''}\ip{c'}{h'} ) \\
& + w_7 ( \ip{\dot{c}'}{\dot{h}'}\ip{c''}{k'} + \ip{\dot{c}'}{\dot{k}'}\ip{c''}{h'} + \ip{\dot{c}'}{\dot{h}'}\ip{c'}{k''} + \ip{\dot{c}'}{\dot{k}'}\ip{c'}{h''} ) \\
& + w_8 \left( \ip{\dot{c}''}{\dot{h}'}\ip{c''}{k'} + \ip{\dot{c}''}{\dot{k}'}\ip{c''}{h'} + \ip{\dot{c}'}{\dot{h}''}\ip{c''}{k'} + \ip{\dot{c}'}{\dot{k}''}\ip{c''}{h'} \right. \\
&\quad\qquad + \left. \ip{\dot{c}''}{\dot{h}'}\ip{c'}{k''} + \ip{\dot{c}''}{\dot{k}'}\ip{c'}{h''} + \ip{\dot{c}'}{\dot{h}''}\ip{c'}{k''} + \ip{\dot{c}'}{\dot{k}''}\ip{c'}{h''} \right) \\
& + w_9 ( \ip{\dot{c}''}{\dot{h}''}\ip{c'}{k'} + \ip{\dot{c}''}{\dot{k}''}\ip{c'}{h'} ) \\
& + t_1 \ip{h'}{k'} 
+ t_2 \left( \ip{h''}{k'} + \ip{h'}{k''} \right)
+ t_3 \ip{\dot{h}}{\dot{k}}
+ t_4 \ip{\dot{h}'}{\dot{k}'}  \\
& + t_5 \left( \ip{\dot{h}''}{\dot{k}'} + \ip{\dot{h}'}{\dot{k}''} \right)
+ t_6 \ip{\dot{h}''}{\dot{k}''} \ud \theta \ud t\,,
\end{align*}
with
\begin{align*}
w_1 &= -a_0 \frac{1}{|c'|}\ip{\dot{c}}{\dot{c}} + a_1 \frac{3}{|c'|^5}\ip{\dot{c}'}{\dot{c}'} + a_2 \frac{63}{|c'|^{11}}\ip{c'}{c''}^2\ip{\dot{c}'}{\dot{c}'} \\
&\qquad - a_2 \frac{70}{|c'|^9}\ip{c'}{c''}\ip{\dot{c}'}{\dot{c}''} + a_2 \frac{15}{|c'|^7}\ip{\dot{c}''}{\dot{c}''}  \,,\\ 
w_2 &= - a_2 \frac{14}{|c'|^9}\ip{c'}{c''}\ip{\dot{c}'}{\dot{c}'} + a_2 \frac{10}{|c'|^7}\ip{\dot{c}'}{\dot{c}''}\,, \quad
w_3 = a_2 \frac{2}{|c'|^7}\ip{\dot{c}'}{\dot{c}'} \,, \quad
w_4 = a_0 \frac{2}{|c'|} \,,\\
w_5 &= -a_1 \frac{2}{|c'|^3} - a_2 \frac{14}{|c'|^9}\ip{c'}{c''}^2\,, \quad
w_6 = a_2 \frac{10}{|c'|^7}\ip{c'}{c''}\,, \quad
w_7 = a_2 \frac{4}{|c'|^7}\ip{c'}{c''} \,,\\
w_8 &= - a_2 \frac{2}{|c'|^5} \,, \quad
w_9 = - a_2 \frac{6}{|c'|^5}\,.
\end{align*}

\subsection*{Acknowledments}

We thank Peter Michor, Jens Gravesen, Hermann Schichl, and the participants of the Math on the Rocks workshop in Grundsund for helpful discussions and valuable comments. We are grateful to the anonymous referees for their careful reading of our manuscript. 

\bibliographystyle{siam}

\end{document}